\documentclass[a4paper,british]{scrartcl}
\usepackage[T1]{fontenc}
\usepackage[latin9]{inputenc}
\usepackage{babel}
\usepackage{verbatim}
\usepackage{prettyref}
\usepackage{mathtools}
\usepackage{amsmath}
\usepackage{amsthm}
\usepackage{amssymb}
\usepackage[unicode=true,pdfusetitle,
 bookmarks=true,bookmarksnumbered=false,bookmarksopen=false,
 breaklinks=false,pdfborder={0 0 1},backref=false,colorlinks=false]
 {hyperref}

\makeatletter

\pdfpageheight\paperheight
\pdfpagewidth\paperwidth

\numberwithin{equation}{section}

\theoremstyle{plain}
\newtheorem{thm}{\protect\theoremname}[section]
\theoremstyle{definition}
\newtheorem{problem}[thm]{\protect\problemname}
\theoremstyle{definition}
\newtheorem*{defn*}{\protect\definitionname}
\theoremstyle{plain}
\newtheorem{prop}[thm]{\protect\propositionname}
\theoremstyle{remark}
\newtheorem{rem}[thm]{\protect\remarkname}
\theoremstyle{plain}
\newtheorem{lem}[thm]{\protect\lemmaname}
\theoremstyle{plain}
\newtheorem{cor}[thm]{\protect\corollaryname}
\theoremstyle{definition}
\newtheorem{defn}[thm]{\protect\definitionname}

\@ifundefined{date}{}{\date{}}
\usepackage{prettyref}

\usepackage{enumerate}
\allowdisplaybreaks

\newcommand{\lin}{\operatorname{lin}}
\newcommand*{\e}{\mathrm{e}}
\renewcommand*{\i}{\mathrm{i}}
\newcommand{\m}{\operatorname{m}}

\newcommand{\R}{\mathbb{R}}
\newcommand{\N}{\mathbb{N}}
\renewcommand{\C}{\mathbb{C}}
\newcommand{\dom}{\operatorname{dom}}
\newcommand{\ran}{\operatorname{ran}}
\newcommand{\spt}{\operatorname{spt}}
\renewcommand{\d}{\,\mathrm{d}}
\renewcommand{\Re}{\operatorname{Re}}
\renewcommand{\Im}{\operatorname{Im}}
\newcommand{\grad}{\operatorname{grad}}
\newcommand{\sgn}{\operatorname{sgn}}
\newcommand{\curl}{\operatorname{curl}}
\newcommand{\dive}{\operatorname{div}}
\renewcommand{\tilde}{\widetilde}

\newrefformat{subsec}{Subsection \ref{#1}}
\newrefformat{prob}{Problem \ref{#1}}
\newrefformat{prop}{Proposition \ref{#1}}
\newrefformat{lem}{Lemma \ref{#1}}
\newrefformat{thm}{Theorem \ref{#1}}
\newrefformat{cor}{Corollary \ref{#1}}
\newrefformat{rem}{Remark \ref{#1}}
\newrefformat{exa}{Example \ref{#1}}
\newrefformat{sub}{Subsection \ref{#1}}
\newrefformat{eq}{(\ref{#1})}

\theoremstyle{definition}
\newtheorem{hyp}{Hypotheses}

\newrefformat{hyp}{Hypotheses \ref{#1}}

\makeatother

\providecommand{\corollaryname}{Corollary}
\providecommand{\definitionname}{Definition}
\providecommand{\lemmaname}{Lemma}
\providecommand{\problemname}{Problem}
\providecommand{\propositionname}{Proposition}
\providecommand{\remarkname}{Remark}
\providecommand{\theoremname}{Theorem}

\begin{document}

\title{Maximal Regularity for Non-Autonomous Evolutionary Equations}

\author{Sascha Trostorff\footnote{Sascha Trostorff, Mathematisches Seminar, CAU Kiel, Germany}\hspace{0.1cm} \& Marcus Waurick\footnote{Marcus
Waurick, Department of Mathematics and Statistics, University of Strathclyde,
Glasgow, Scotland}
}
\maketitle
\begin{abstract}
We discuss the issue of maximal regularity for evolutionary equations
with non-autonomous coefficients. Here evolutionary equations are
abstract partial-differential algebraic equations considered in Hilbert
spaces. The catch is to consider time-dependent partial differential
equations in an exponentially weighted Hilbert space. In passing,
one establishes the time derivative as a continuously invertible,
normal operator admitting a functional calculus with the Fourier--Laplace
transformation providing the spectral representation. Here, the main
result is then a regularity result for well-posed evolutionary equations
solely based on an assumed parabolic-type structure of the equation
and estimates of the commutator of the coefficients with the square
root of the time derivative. We thus simultaneously generalise available
results in the literature for non-smooth domains. Examples for equations
in divergence form, integro-differential equations, perturbations
with non-autonomous and rough coefficients as well as non-autonomous
equations of eddy current type are considered.
\end{abstract}
\textbf{MSC2020}: 35B65, 35R20, 35K90, 26A33

\textbf{Keywords}: Non-autonomous maximal regularity, Evolutionary equations,
Lions' problem, Commutator estimates, Riemann--Liouville fractional
derivative

\section{Introduction}

If one considers partial differential equations depending on time
as an equation in space-time the following problem of maximal regularity
arises naturally. For the sake of the argument, let $\mathcal{H}$
be a Hilbert space modelling space-time and let $D$ and $A$ be two
closed, densely defined (unbounded) operators, where the former contains
the temporal and the latter the spatial derivative(s). Abstractly
spoken, the PDE in question then may look like as follows:
\[
Du+Au=f
\]
for some right-hand side $f\in\mathcal{H}$. In particular, when hyperbolic
type problems are concerned (think of the transport equation or the
wave equation), one cannot expect that for any $f\in\mathcal{H}$
(usually an $L_{2}$-type space) the solution $u$ to belong to both
$\dom(D)$ and $\dom(A)$. In general one can only hope for $u\in\dom(\overline{D+A}),$
thus deeming the above equation to be true only in some \emph{generalised}
sense. A first example, where it is possible to show that $u$ belongs
to the individual domains given any $f\in\mathcal{H}$ is when $D=\partial_{t}$
and $A=-\Delta_{D}$ (Laplacian with Dirichlet boundary conditions
on some open $\Omega\subseteq\mathbb{R}^{n}$) and $\mathcal{H}=L_{2}(0,T;L_{2}(\Omega))$
and $u$ is assumed to satisfy homogeneous initial conditions. Then,
the solution $u$ indeed belongs to $H^{1}(0,T;L_{2}(\Omega))\cap L_{2}(0,T;\dom(\Delta_{D}))$,
see e.g.~\cite{Lions1961} or below. Traditionally, the method of
choice to derive such a regularity result first establishes well-posedness
of the equation at hand via bilinear forms and afterwards analysing
the problem in terms of the associated generator. Quite naturally
and generalising the above situation considerably, Lions raised the
following problem (see \cite[p. 68]{Lions1961}):
\begin{problem}
\label{prob:Lions}Let $V$ and $H$ be Hilbert spaces such that $V\hookrightarrow H$
continuously and densely. Let $a\colon[0,T]\times V\times V\to\mathbb{C}$
be such that $a(t,\cdot,\cdot)$ is sesquilinear, satisfying suitable
boundedness coercivity and measurability conditions thus defining
$\mathfrak{A}(t)$ via $\langle\mathfrak{A}(t)x,y\rangle_{V,V^{*}}\coloneqq a(t,x,y)$,
$t\in[0,T]$. Let $f\in L_{2}(0,T;H)$ be given. Then there exists
a unique solution $u\in H^{1}(0,T;V^{*})\cap L_{2}(0,T;V)$ of
\[
u'(t)+\mathfrak{A}(t)u(t)=f(t)\quad u(0)=0.
\]
 The question now is under which conditions on $a$, do we actually
have $u\in H^{1}(0,T;H)$?
\end{problem}

The latter problem indeed fits into the above abstract perspective
for $D=\partial_{t}$ with domain $H^{1}(0,T;H)$ with Dirichlet boundary
conditions at $0$ and $\tilde{A}=\mathfrak{A}\colon\dom(\mathfrak{A})\subseteq L_{2}(0,T;H)\to L_{2}(0,T;H),u\mapsto(t\mapsto\mathfrak{A}(t)u(t))$
with maximal domain. \prettyref{prob:Lions} has a long history and
has rather recently gained some renewed attention. For the latest
developments, we refer to the survey article in \cite{Arendt2017},
to \cite{Achache2019} and its introduction. We recall here that Hölder
continuity for $a$ (and particularly the Hölder exponent $1/2$)
with respect to time in a suitable sense plays a crucial role, see
e.g.~\cite{Ouhabaz2010,Fackler2017} for a positive and a negative
result, respectively. 

The available results in the literature up to this point consider
explicit Cauchy problems similar to the one in \prettyref{prob:Lions}.
Thus, in any case, the complexity of the problem is contained in the
form $a$ (or in the operator $\mathfrak{A}$). 

In this article we set a different focus and try to keep the operator
containing the spatial derivatives (i.e.~$\mathfrak{A}$) as simple
as possible and move the complexity over to the time derivative. The
rationale behind this is the notion of so-called evolutionary equations,
invented in \cite{Picard2009} and rather self-contained discussed
in \cite{Picard2020,seifert2020evolutionary}. More precisely, we
consider equations of the form
\[
\left(\partial_{t}\mathcal{M}+\mathcal{N}+A\right)U=F,
\]
where $f$ belongs to an exponentially weighted $L_{2}$-space, $A$
is an unbounded skew-selfadjoint operator solely acting with respect
to the spatial variables and $\mathcal{M}$ and $\mathcal{N}$ are
suitable bounded linear operators in space-time. The solution theory
developed in \cite{W16_H,Waurick2015} asserts that under suitable
positive definiteness conditions imposed on $\partial_{t}\mathcal{M}+\mathcal{N}$,
one indeed has that $\overline{\left(\partial_{t}\mathcal{M}+\mathcal{N}+A\right)}$
is continuously invertible. In the framework of evolutionary equations,
the maximal regularity problem then reads as follows.
\begin{problem}
\label{prob:LionsEE}Given $\partial_{t}\mathcal{M}+\mathcal{N}$
satisfies the appropriate positive definiteness conditions and $A$
skew-selfadjoint in order that $\overline{\left(\partial_{t}\mathcal{M}+\mathcal{N}+A\right)}$
is continuously invertible, what are the additional conditions on
$\mathcal{M}$ and $\mathcal{N}$ (and the right-hand side $F$) such
that 
\[
\overline{\left(\partial_{t}\mathcal{M}+\mathcal{N}+A\right)}^{-1}F=\left(\partial_{t}\mathcal{M}+\mathcal{N}+A\right)^{-1}F\ ?
\]
\end{problem}

We emphasise that even in the time-independent case \prettyref{prob:Lions}
and \prettyref{prob:LionsEE} are rather different types of questions.
In fact, since $A$ is skew-selfadjoint in \prettyref{prob:LionsEE},
the choices $\mathcal{M}=1$ and $\mathcal{N}=0$ for $F\notin H^{1}$
with respect to time do not lead to $\left(\overline{\partial_{t}+A}\right)^{-1}F=\left(\partial_{t}+A\right)^{-1}F$,
if $A$ is unbounded. As it will be obvious in the next example for
a solution of \prettyref{prob:LionsEE} one is particularly interested
in cases where $F$ belongs to spaces not as smooth as $H^{1}$. 

In the autonomous case, \prettyref{prob:LionsEE} has been addressed
in \cite{PTW2017_maxreg}. The conditions derived describe a parabolic
type evolutionary equation in an abstract manner. Indeed, one assume
that there exists a densely defined closed linear operator $C$ (acting
in the spatial variables only) such that 
\[
A=\left(\begin{array}{cc}
0 & -C^{*}\\
C & 0
\end{array}\right).
\]
Moreover, one has that 
\[
\mathcal{M}=\left(\begin{array}{cc}
\mathcal{M}_{00} & 0\\
0 & 0
\end{array}\right)\text{ and }\mathcal{N}=\left(\begin{array}{cc}
\mathcal{N}_{00} & \mathcal{N}_{01}\\
\mathcal{N}_{10} & \mathcal{N}_{11}
\end{array}\right)
\]
with $\mathcal{N}_{11}$ satisfying an additional positive definiteness
condition. The standard case of the heat equation $\partial_{t}u-\Delta_{D}u=f$
mentioned above is then recovered by putting $q=-\grad_{0}u$ (gradient
subject to homogeneous Dirichlet boundary conditions) and considering
\[
\left(\partial_{t}\left(\begin{array}{cc}
1 & 0\\
0 & 0
\end{array}\right)+\left(\begin{array}{cc}
0 & 0\\
0 & 1
\end{array}\right)+\left(\begin{array}{cc}
0 & \dive\\
\grad_{0} & 0
\end{array}\right)\right)\left(\begin{array}{c}
u\\
q
\end{array}\right)=\left(\begin{array}{c}
f\\
0
\end{array}\right).
\]
Then, indeed, by the main result of \cite{PTW2017_maxreg}, one has
\begin{align*}
 & \overline{\left(\partial_{t}\left(\begin{array}{cc}
1 & 0\\
0 & 0
\end{array}\right)+\left(\begin{array}{cc}
0 & 0\\
0 & 1
\end{array}\right)+\left(\begin{array}{cc}
0 & \dive\\
\grad_{0} & 0
\end{array}\right)\right)}^{-1}\left(\begin{array}{c}
f\\
0
\end{array}\right)\\
 & =\left(\partial_{t}\left(\begin{array}{cc}
1 & 0\\
0 & 0
\end{array}\right)+\left(\begin{array}{cc}
0 & 0\\
0 & 1
\end{array}\right)+\left(\begin{array}{cc}
0 & \dive\\
\grad_{0} & 0
\end{array}\right)\right)^{-1}\left(\begin{array}{c}
f\\
0
\end{array}\right)
\end{align*}
leading to the maximal regularity result mentioned at the beginning
for $F=(f,0)$ with $f\in L_{2}(0,T;L_{2}(\Omega))$ only.

Even though the class of equations treated in \cite{PTW2017_maxreg}
particularly contains integro-differential equations rendering rather
different equations to enjoy maximal regularity, the case of the heat
equation with non-symmetric but time-independent coefficients could
not be treated with the methods developed there. 

In this article we shall enlarge the class of coefficients $\mathcal{M}$
and $\mathcal{N}$ considerably leading to the equality highlighted
in \prettyref{prob:LionsEE}. In particular, this class will involve
non-symmetric conductivities in the case of the heat equation. What
is more, we shall show that the conditions might be weaker than the
conditions derived in both \cite{Auscher_Egert2016} or \cite{Dier_Zacher2017}
if applied to divergence form problems. Since we do not consider bilinear
forms as our central object of study, we do not invoke the Kato square
root property explicitly, which proved instrumental in the main result
in \cite{Achache2019}. In particular, our methods also apply irrespective
of the regularity of the considered underlying domains of the exemplarily
considered divergence form problems. Another key difference to the
results for non-autonomous maximal regularity available in the literature
is the possibility of the variable operator coefficient $\mathcal{M}$,
which permits us to consider integro-differential equations with the
same approach as classical Cauchy problems in divergence form. Moreover,
the operator coefficient $\mathcal{N}$ permits the introduction of
rough (in time) lower order terms. Before we present a plan of our
paper, we shortly describe the two main results and instrumental techniques
used in the present article. 

\prettyref{thm:main}, our first main result on maximal regularity
of evolutionary equations, in rough terms can be described as follows:
Well-posedness in $L_{2}$ and $H^{1/2}$ together with a parabolic
structure of $\mathcal{M},\mathcal{N}$ and $A$ imply maximal regularity
in the sense of \prettyref{prob:LionsEE} for $F=(f,g)\in L_{2}\times H^{1/2}$,
which in the standard heat equation case is satisfied as $g=0$ anyway.
For a proof of \prettyref{thm:main}, the framework of evolutionary
equations is particularly helpful since $\partial_{t}$ is continuously
invertible and normal yielding a handy decription of $H^{1/2}$ by
the functional calculus for $\partial_{t}$. The functional calculus
is provided with the help of the Fourier--Laplace transformation.
Note that the application of this functional calculus naturally leads
to the fractional Riemann--Liouville derivative, see also \cite{PTW2015_fractional,Diethelm2019}.
In applications, the conditions on the parabolic structure and the
well-posedness in $L_{2}$ are rather easy to show. The assumed positive
definitenss in $H^{1/2}$ leading to the respective well-posedness
result might be rather difficult to obtain, though. We emphasise,
however, that in addition to the various positive definiteness estimates,
we only need to assume that the involved coefficient operators $\mathcal{M}$
and $\mathcal{N}$ are bounded linear operators in $H^{1/2}$ thus
leaving this space invariant. In particular, \emph{no} bounded commutator
assumptions need to be imposed suggesting room for improvement along
the lines of the low regularity assumed for the coefficients in \cite{Achache2019}.
We shall not follow up on this but rather assume stronger commutator
assumptions on $\mathcal{M}$ and $\mathcal{N}$ with $\partial_{t}^{-1}$
and $\partial_{t}^{1/2}$, respectively, confirming the particular
role of commutator estimates for maximal regularity already observed
\cite{Auscher_Egert2016,Dier_Zacher2017}. Our second main theorem
on maximal regularity of evolutionary equations (\prettyref{thm:max_reg_commuator})
imposes the same parabolic structure assumption and well-posedness-in-$L_{2}$-requirement
as \prettyref{thm:main}. The conditions on the commutators then lead
to the asked for well-posedness in $H^{1/2}$ of \prettyref{thm:main}
via a perturbation argument. 

For divergence form problems, the assumptions in \prettyref{thm:max_reg_commuator}
are implied by the fractional Sobolev (or BMO)-regularity properties
imposed in either \cite{Auscher_Egert2016} or \cite{Dier_Zacher2017}.
This provides a way of classifying the a priori not comparable conditions
in \cite{Auscher_Egert2016} and \cite{Dier_Zacher2017}. Furthermore,
we recover an analogous regularity phenomenon first observed in \cite{Dier_Zacher2017}
and confirmed in \cite{Auscher_Egert2016} of the solution belonging
to $H^{1/2}$ -time regularity taking values in the form domain. 

In the next section, we recall the framework of evolutionary equations
and highlight the main ingredients of the non-autonomous solution
theory in $L_{2}$ as well as some facts of the (time) derivative
established in vector-valued exponentially weighted $L_{2}$-spaces.
This particularly includes the spectral representation and the accompanying
functional calculus. In \prettyref{sec:The-solution-theory H12},
we provide a necessary new technical result, \prettyref{thm:sol_theory_1/2},
which contains a solution theory for evolutionary equations in $H^{1/2}$.
Our first main result is presented and proved in \prettyref{sec:Maximal-regularity-forEE}.
The corresponding perturbation result with the mentioned commutator
assumptions is presented in \prettyref{sec:Applications}. Also, with
a focus on operator-valued multiplication operators, we analyse the
commutator condition imposed in \prettyref{thm:max_reg_commuator}
a bit more closely. We provide a proof of the results in \cite{Auscher_Egert2016,Dier_Zacher2017}
for divergence form problems with our methods in \prettyref{subsec:Divergence-form-equations}.
An example for integro-differential equations being a non-autonomous
variant of some equations considered in \cite{Trostorff2015} is presented
in \prettyref{subsec:Maximal-regularity-integro}. The last application
of our abstract findings is concerned with the (non-autonomous) eddy
current approximation for Maxwell's equations in \prettyref{subsec:Maxwell's-equations}.
We provide a small conclusion in \prettyref{sec:Conclusion}.

\section{The Framework\label{sec:The-Framework}}

We recall the framework of evolutionary equations. For more details
and the proofs we refer to \cite{Picard_McGhee2011,seifert2020evolutionary,W16_H}.
We start with the underlying Hilbert space setting and the definition
of the time derivative operator. 
\begin{defn*}
For $\rho\geq0$ we define the space 
\[
L_{2,\rho}(\R;H)\coloneqq\{f:\R\to H\,;\,f\text{ Bochner-measurable},\,\int_{\R}\|f(t)\|^{2}\e^{-2\rho t}\d t<\infty\},
\]
where we as usual identify functions which are equal almost everywhere.
This space is clearly a Hilbert space with respect to the inner product
\[
\langle f,g\rangle_{\rho,0}\coloneqq\int_{\R}\langle f(t),g(t)\rangle_{H}\e^{-2\rho t}\d t\quad(f,g\in L_{2,\rho}(\R;H)).
\]
Moreover, we define the operator $\partial_{t,\rho}:\dom(\partial_{t,\rho})\subseteq L_{2,\rho}(\R;H)\to L_{2,\rho}(\R;H)$
as the closure of the operator 
\[
C_{c}^{\infty}(\R;H)\subseteq L_{2,\rho}(\R;H)\to L_{2,\rho}(\R;H),\quad\phi\mapsto\phi',
\]
where $C_{c}^{\infty}(\R;H)$ denotes the space of arbitrarily differentiable
functions having compact support attaining values in $H$. Finally,
we define the \emph{Fourier--Laplace transformation }$\mathcal{L}_{\rho}:L_{2,\rho}(\R;H)\to L_{2}(\R;H)$
as the continuous extension of the mapping 
\[
C_{c}^{\infty}(\R;H)\subseteq L_{2,\rho}(\R;H)\to L_{2}(\R;H),\quad\phi\mapsto\left(t\mapsto\frac{1}{\sqrt{2\pi}}\int_{\R}\e^{-(\i t+\rho)s}\phi(s)\d s\right).
\]
\end{defn*}
We collect some properties of the so introduced operators.
\begin{prop}
\label{prop:elementary facts td}Let $\rho\geq0.$ 

\begin{enumerate}[(a)]

\item The operator $\partial_{t,\rho}$ is normal with $\Re\partial_{t,\rho}=\rho.$
Thus, in particular $\partial_{t,\rho}^{-1}\in L(L_{2,\rho}(\R;H))$
with $\|\partial_{t,\rho}^{-1}\|\leq\frac{1}{\rho}$ if $\rho\ne0.$
Moreover, for $\rho\ne0$ 
\[
\left(\partial_{t,\rho}^{-1}f\right)(t)=\int_{-\infty}^{t}f(s)\d s\quad(t\in\R,f\in L_{2,\rho}(\R;H)).
\]

\item The operator $\mathcal{L}_{\rho}$ is unitary and 
\[
\mathcal{L}_{\rho}\partial_{t,\rho}=(\i\m+\rho)\mathcal{L}_{\rho},
\]
where $\m:\dom(\m)\subseteq L_{2}(\R;H)\to L_{2}(\R;H)$ is given
by $\left(\m f\right)(t)=tf(t)$ for $t\in\R$ and $f\in\dom(\m)$
with maximal domain; that is, 
\[
\dom(\m)=\{f\in L_{2}(\R;H)\,;\,(t\mapsto tf(t))\in L_{2}(\R;H)\}.
\]
 In particular $\sigma(\partial_{t,\rho})=\{\i t+\rho\,;\,t\in\R\}$. 

\item As operators in $L_{2,\rho}(\R;H)$ we have 
\[
\partial_{t,\rho}^{*}=-\partial_{t,\rho}+2\rho
\]

\end{enumerate}
\end{prop}

With the help of the unitary equivalence of the operators $\partial_{t,\rho}$
and $\i\m+\rho$ we can also define derivatives of fractional order
(see \cite{PTW2015_fractional,seifert2020evolutionary,Diethelm2019}).
\begin{prop}
\label{prop:fractional_integral}Let $\rho>0$ and $\alpha\in\R$
and set 
\[
\partial_{t,\rho}^{\alpha}\coloneqq\mathcal{L}_{\rho}^{\ast}(\i\m+\rho)^{\alpha}\mathcal{L}_{\rho}.
\]
Then $\partial_{t,\rho}^{\alpha}$ is densely defined and closed on
$L_{2,\rho}(\R;H)$ and if $\alpha\leq0,$ it is bounded with $\|\partial_{t,\rho}^{\alpha}\|\leq\frac{1}{\rho^{\alpha}}.$
Moreover, for $\alpha>0$ we have $\Re\partial_{t,\rho}^{\alpha}\geq\rho^{\alpha}$
and the operator $\partial_{t,\rho}^{-\alpha}$ is given by 
\[
\left(\partial_{t,\rho}^{-\alpha}f\right)(t)=\frac{1}{\Gamma(\alpha)}\int_{-\infty}^{t}(t-s)^{\alpha-1}f(s)\d s\quad(t\in\R,f\in L_{2,\rho}(\R;H)).
\]
\end{prop}

With the help of these operators, we can define the fractional Sobolev
spaces with respect to the exponentially weighted Lebesgue-measure.
\begin{defn*}
Let $\rho>0$ and $\alpha\geq0.$ Then we set 
\[
H_{\rho}^{\alpha}(\R;H)\coloneqq\dom(\partial_{t,\rho}^{\alpha})
\]
and equip it with the norm (note that $\partial_{t,\rho}^{\alpha}$
is injective)
\[
\|u\|_{\rho,\alpha}\coloneqq\|\partial_{t,\rho}^{\alpha}u\|_{\rho,0}\quad(u\in H_{\rho}^{\alpha}(\R;H)).
\]
\end{defn*}
The following proposition is an immediate consequence of the definitions
above.
\begin{prop}
Let $\rho>0$ and $\alpha\geq0$. Then the following statements hold

\begin{enumerate}[(a)]

\item For each $0\leq\beta\leq\alpha$ the operator $\partial_{t,\rho}^{\beta}:H_{\rho}^{\alpha}(\R;H)\to H_{\rho}^{\alpha-\beta}(\R;H)$
is unitary.

\item The operator $\mathcal{L}_{\rho}:H_{\rho}^{\alpha}(\R;H)\to H^{\alpha}(\i\m+\rho)$
is unitary. Here, $H^{\alpha}(\i\m+\rho)=\{u\in L_{2}(\R;H)\,;\,\left(t\mapsto(\i t+\rho)^{\alpha}u(t)\right)\in L_{2}(\R;H)\}$
equipped with the norm $\|u\|_{H^{\alpha}(\i\m+\rho)}\coloneqq\|(\i\m+\rho)^{\alpha}u\|_{L_{2}(\R;H)}.$

\end{enumerate}
\end{prop}

\begin{rem}
\label{rem:interpolation}(a) For $\theta\in]0,1[$ the space $H_{\rho}^{\theta}(\R;H)$
can also be obtained by complex interpolation. More precisely, we
have
\begin{equation}
H_{\rho}^{\theta}(\R;H)=(L_{2,\rho}(\R;H),H_{\rho}^{1}(\R;H))_{\theta}\label{eq:HthetaIS}
\end{equation}
isometrically. For the theory of interpolation spaces we refer to
\cite{Bergh_Lofstrom1976,Lunardi2018}. To see \prettyref{eq:HthetaIS},
we consider the unitarily transformed space $H^{\theta}(\i\m+\rho)$
first and show
\[
H^{\theta}(\i\m+\rho)=(L_{2}(\R;H),H^{1}(\i\m+\rho))_{\theta}
\]
isometrically. Indeed, for $u\in H^{\theta}(\i\m+\rho)$ we define
the function 
\begin{align*}
f_{u}:S & \to L_{2}(\R;H)\\
z & \mapsto(t\mapsto u(t)|\i t+\rho|^{\theta-z}).
\end{align*}
Here $S\coloneqq\{z\in\C\,;\,\Re z\in[0,1]\}.$ Note that $f_{u}$
is well-defined and bounded, since 
\[
\int_{\R}|f_{u}(z)(t)|^{2}\d t=\int_{\R}|u(t)(\i t+\rho)^{\theta}|^{2}|\i t+\rho|^{-2\Re z}\d t\leq\|u\|_{H^{\theta}(\i\m+\rho)}^{2}\max\{\rho^{-2},1\}.
\]
Moreover, $f_{u}$ is holomorphic in the interior of $S$, $f_{u}(\i\xi)\in L_{2}(\R;H)$
with 
\[
\|f_{u}(\i\xi)\|_{L_{2}(\R;H)}=\|u\|_{H^{\theta}(\i\m+\rho)},
\]
 and $f_{u}(\i\xi+1)\in H^{1}(\i\m+\rho)$ with 
\[
\|f_{u}(\i\xi+1)\|_{H^{1}(\i\m+\rho)}=\|u\|_{H^{\theta}(\i\m+\rho)}
\]
 for each $\xi\in\R.$ Since $f_{u}(\theta)=u$ this implies $u\in(L_{2}(\R;H),H_{\rho}^{1}(\i\m+\rho))_{\theta}$
with 
\[
\|u\|_{(L_{2}(\R;H),H_{\rho}^{1}(\i\m+\rho))_{\theta}}\leq\|u\|_{H^{\theta}(\i\m+\rho)}.
\]
For showing the converse inclusion and norm inequality, let $u\in(L_{2}(\R;H),H_{\rho}^{1}(\i\m+\rho))_{\theta}$
and $g:S\to L_{2}(\R;H)+H^{1}(\i\m+\rho)$ continuous, holomorphic
in the interior of $S$ and bounded with $g(\i\xi)\in L_{2}(\R;H)$
and $g(\i\xi+1)\in H^{1}(\i\m+\rho)$ such that $g(\theta)=u.$ To
prove that $u\in H^{\theta}(\i\m+\rho)$ it suffices to show that
\[
H_{c}^{\theta}(\i\m+\rho)\ni v\mapsto\int_{\R}\langle u(t),v(t)\rangle_{H}|\i t+\rho|^{2\theta}\d t
\]
defines a bounded functional on $H^{\theta}(\i\m+\rho)$, where $H_{c}^{\theta}(\i\m+\rho)$
denotes the elements in $H^{\theta}(\i\m+\rho)$ having compact support.
For this, let $v\in H_{c}^{\theta}(\i\m+\rho)$ and set $f_{v}:S\to L_{2}(\R;H)$
as above and consider the function 
\begin{align*}
F:S & \to\C\\
z & \mapsto\int_{\R}\langle g(z)(t),f_{v}(z)(t)\rangle_{H}|\i t+\rho|^{2z}\d t.
\end{align*}
Note that this integral is well-defined since $f_{v}(z)\in L_{2}(\R;H)$
is compactly supported for each $z\in S$. The so defined $F$ is
continuous, holomorphic in the interior of $S$ and bounded and thus,
\[
|F(z)|\leq\sup_{\xi\in\R}\{|F(\i\xi)|,|F(\i\xi+1)|\}\quad(z\in S)
\]
by the maximum principle. For $\xi\in\R$ we estimate 
\begin{align*}
|F(\i\xi)| & \leq\|g(\i\xi)\|_{L_{2}(\R;H)}\|f_{v}(\i\xi)\|_{L_{2}(\R;H)}\\
 & =\|g(\i\xi)\|_{L_{2}(\R;H)}\|v\|_{H^{\theta}(\i\m+\rho)}\text{ and }\\
|F(\i\xi+1)| & \leq\|g(\i\xi+1))\|_{H^{1}(\i\m+\rho)}\|f_{v}(\i\xi+1)\|_{H^{1}(\i\m+\rho)}\\
 & =\|g(\i\xi+1))\|_{H^{1}(\i\m+\rho)}\|v\|_{H^{\theta}(\i\m+\rho)}.
\end{align*}
Thus, we obtain
\begin{align*}
\left|\int_{\R}\langle u(t),v(t)\rangle_{H}|\i t+\rho|^{2\theta}\d t\right| & =|F(\theta)| \\
 &\leq\sup_{\xi\in\R}\{\|g(\i\xi)\|_{L_{2}(\R;H)},\|g(\i\xi+1))\|_{H^{1}(\i\m+\rho)}\}\|v\|_{H^{\theta}(\i\m+\rho)}.
\end{align*}
Taking now the infimum over all $g$ with the desired properties,
we get 
\[
\left|\int_{\R}\langle u(t),v(t)\rangle_{H}|\i t+\rho|^{2\theta}\d t\right|\leq\|u\|_{(L_{2}(\R;H),H^{1}(\i\m+\rho))_{\theta}}\|v\|_{H^{\theta}(\i\m+\rho)},
\]
which yields $u\in H^{\theta}(\i\m+\rho)$ with $\|u\|_{H^{\theta}(\i\m+\rho)}\leq\|u\|_{(L_{2}(\R;H),H^{1}(\i\m+\rho))_{\theta}}.$
\\
Finally, using that $\mathcal{L}_{\rho}:H_{\rho}^{\theta}(\R;H)\to H^{\theta}(\i\m+\rho)$
is unitary for each $\theta\in[0,1],$ we obtain the assertion.

(b) Let $\mathcal{M}\in L(L_{2,\rho}(\mathbb{R};H))\cap L(H_{\rho}^{1}(\mathbb{R};H)).$
Then, for all $\theta\in[0,1]$, by (a), $\mathcal{M}\in L(H_{\rho}^{\theta}(\mathbb{R};H))$;
see also \cite[Theorem 2.6]{Lunardi2018}.
\end{rem}

\begin{rem}
\label{rem:FourierIP}(a) For the next result, we briefly recall that
by complex interpolation and Plancherel's theorem, we have that the
Fourier transformation extends to be a continuous operator 
\[
\mathcal{F}\colon L_{p'}(\R;H)\to L_{p}(\R;H),
\]
where $\frac{1}{p'}+\frac{1}{p}=1$ with $1<p'<2<p<\infty$. Indeed,
this follows from the fact that $\mathcal{F}\colon L_{2}(\R;H)\to L_{2}(\R;H)$
and $\mathcal{F}\colon L_{1}(\mathbb{R};H)\to L_{\infty}(\R;H)$ are
unitary and continuous, respectively. Note that this applies verbatim
to $\mathcal{F}^{*}$. 

(b) Furthermore, we recall the following version of Hölder's inequality:
If $f\in L_{p}(\R)$ and $g\in L_{q}(\R;X)$, $X$ a Banach space,
and $\frac{1}{p}+\frac{1}{q}=\frac{1}{r}$ for some $p,q,r\in[1,\infty],$
then $t\mapsto f(t)g(t)\in L_{r}(\R;X)$ and
\[
\|fg\|_{L_{r}}\leq\|f\|_{L_{p}}\|g\|_{L_{q}}.
\]
\end{rem}

\begin{lem}
\label{lem:Sobolev}Let $\rho>0$ and $\alpha\in]0,1/2]$. Then for
each $p\in[2,\frac{2}{1-2\alpha}[$ (where we set $\frac{2}{0}\coloneqq\infty$)
$H_{\rho}^{\alpha}(\R;H)\hookrightarrow L_{p,\rho}(\R;H)$, where
\[
L_{p,\rho}(\R;H)\coloneqq\{f:\R\to H\,;\,f\text{ measurable, }\int_{\R}\|f(t)\|^{p}\e^{-p\rho t}\d t<\infty\}
\]
equipped with the obvious norm, denoted by $\|\cdot\|_{L_{p,\rho}}$. 
\end{lem}

\begin{proof}
For $p=2$ there is nothing to show. Let $p\in]2,\frac{2}{1-2\alpha}[$
and $p'\in]\frac{2}{1+2\alpha},2[$ denote the conjugate exponent
to $p$; i.e. $\frac{1}{p}+\frac{1}{p'}=1$. From \prettyref{rem:FourierIP},
we know that $\mathcal{F}^{*}:L_{p'}\to L_{p}$ is continuous. Hence,
for $u\in C_{c}^{\infty}(\mathbb{R};H)$ we estimate
\begin{align*}
\|u\|_{L_{p,\rho}} & =\|\e^{-\rho\cdot}u\|_{L_{p}}\\
 & =\|\mathcal{F}^{*}\mathcal{L}_{\rho}u\|_{L_{p}}\\
 & \lesssim\|\mathcal{L}_{\rho}u\|_{L_{p'}}.
\end{align*}
Next, let $q\coloneqq\frac{2p'}{2-p'}.$ Then by Hölder's inequality
\begin{align*}
\|\mathcal{L}_{\rho}u\|_{L_{p'}} & =\|(\i\m+\rho)^{-\alpha}(\i\m+\rho)^{\alpha}\mathcal{L}_{\rho}u\|_{L_{p'}}\\
 & \leq\|(\i\m+\rho)^{-\alpha}\|_{L_{q}}\|(\i\m+\rho)^{\alpha}\mathcal{L}_{\rho}u\|_{L_{2}}.
\end{align*}
Note that 
\[
\|(\i\m+\rho)^{-\alpha}\|_{q}=\left(\int_{\R}\frac{1}{(t^{2}+\rho^{2})^{-\frac{\alpha q}{2}}}\d t\right)^{\frac{1}{q}}<\infty,
\]
since $q=\frac{2p'}{2-p'}=\frac{2}{(2/p')-1}>\frac{1}{\alpha}$ and
hence, the claim follows.
\end{proof}
The next statement contains an approximation result, which has been
employed in \cite{Waurick2015,PTW2017_maxreg} for the particular
case $\alpha=0$. To have a corresponding result for the case when
$\alpha>0$ (and particularly when $\alpha=1/2$), will turn out to
be useful in the next section, where we provide a well-posedness result
for evolutionary equations in $H_{\rho}^{1/2}(\R;H)$.
\begin{lem}
\label{lem:approx}Let $\rho>0$ and $\alpha\geq0.$ We consider the
time derivative operator on $H_{\rho}^{\alpha}(\R;H)$; that is, 
\[
\partial_{t,\rho}:H_{\rho}^{\alpha+1}(\R;H)\subseteq H_{\rho}^{\alpha}(\R;H)\to H_{\rho}^{\alpha}(\R;H).
\]
Then for each $\varepsilon>0$ the operator $1+\varepsilon\partial_{t,\rho}$
is continuously invertible on $H_{\rho}^{\alpha}(\R;H)$ and $(1+\varepsilon\partial_{t,\rho})^{-1}\to1$
strongly in $H_{\rho}^{\alpha}(\R;H)$ as $\varepsilon\to0.$ 
\end{lem}

\begin{proof}
For $u\in H_{\rho}^{\alpha+1}(\R;H)$ we have that 
\[
\Re\langle(1+\varepsilon\partial_{t,\rho})u,u\rangle_{\rho,\alpha}=\|u\|_{\rho,\alpha}^{2}+\varepsilon\Re\langle\partial_{t,\rho}\partial_{t,\rho}^{\alpha}u,\partial_{t,\rho}^{\alpha}u\rangle_{\rho,0}\geq\|u\|_{\rho,\alpha}^{2}.
\]
Thus, $(1+\varepsilon\partial_{t,\rho})$ is injective, posseses a
closed range and its inverse (defined on the range) is continuous
with operator norm bounded by 1. Thus, to prove the continuous invertibility,
we have to show that $\ran(1+\varepsilon\partial_{t,\rho})$ is dense
in $H_{\rho}^{\alpha}(\R;H)$. For doing so, we first compute the
adjoint of $\partial_{t,\rho}$. For elements $v,w\in H_{\rho}^{\alpha}(\R;H)$
we have that $v\in\dom(\partial_{t,\rho}^{\ast})$ with $\partial_{t,\rho}^{\ast}v=w$
if and only if 
\[
\langle\partial_{t,\rho}u,v\rangle_{\rho,\alpha}=\langle u,w\rangle_{\rho,\alpha}\quad(u\in H_{\rho}^{\alpha+1}(\R;H)).
\]
The latter is equivalent to 
\[
\langle\partial_{t,\rho}x,\partial_{t,\rho}^{\alpha}v\rangle_{\rho,0}=\langle x,\partial_{t,\rho}^{\alpha}w\rangle_{\rho,0}\quad(x\in H_{\rho}^{1}(\R;H)),
\]
which in turn is equivalent to $\partial_{t,\rho}^{\alpha}v\in H_{\rho}^{1}(\R;H)$
and $\partial_{t,\rho}^{\alpha}w=(-\partial_{t,\rho}+2\rho)\partial_{t,\rho}^{\alpha}v,$
where we used $\partial_{t,\rho}^{*}=-\partial_{t,\rho}+2\rho$ in
$L_{2,\rho}(\R;H)$, see \prettyref{prop:elementary facts td}. Thus,
\[
\partial_{t,\rho}^{\ast}=-\partial_{t,\rho}+2\rho,
\]
where both operators are considered as operators on $H_{\rho}^{\alpha}(\R;H).$
Thus, 
\[
\Re\langle(1+\varepsilon\partial_{t,\rho}^{\ast})u,u\rangle_{\rho,\alpha}=\|u\|_{\rho,\alpha}^{2}+\varepsilon\Re\langle u,\partial_{t,\rho}u\rangle_{\rho,\alpha}\geq\|u\|_{\rho,\alpha}^{2}\quad(u\in H_{\rho}^{\alpha+1}(\R;H)=\dom(\partial_{t,\rho}^{\ast}))
\]
and hence, $1+\varepsilon\partial_{t,\rho}^{\ast}$ is injective,
which shows the density of $\ran(1+\varepsilon\partial_{t,\rho})$
in $H_{\rho}^{\alpha+1}(\R;H).$ To prove the strong convergence,
it suffices to show the convergence for elements in $H_{\rho}^{\alpha+1}(\R;H),$
since $\sup_{\varepsilon>0}\|(1+\varepsilon\partial_{t,\rho})^{-1}\|_{L(H_{\rho}^{\alpha}(\mathbb{R};H))}\leq1$
by what we have shown above. For $u\in H_{\rho}^{\alpha+1}(\R;H)$
we compute 
\begin{align*}
\|(1+\varepsilon\partial_{t,\rho})^{-1}u-u\|_{\rho,\alpha} & =\|(1+\varepsilon\partial_{t,\rho})^{-1}(u-(u+\varepsilon\partial_{t,\rho}u))\|_{\rho,\alpha}\\
 & \leq\varepsilon\|u\|_{\rho,\alpha+1}\to0\quad(\varepsilon\to0).\tag*{\qedhere}
\end{align*}
\end{proof}
We conclude this section, by citing the main result of \cite{Waurick2015}.
\begin{thm}[{\cite[Theorem 3.4]{Waurick2015}}]
\label{thm:sol_theory_basic} Let $\rho>0$ and $\mathcal{M},\mathcal{N}\in L(L_{2,\rho}(\R;H)).$
Moreover, assume there exists $\mathcal{M}'\in L(L_{2,\rho}(\R;H))$
such that 
\[
\mathcal{M}\partial_{t,\rho}\subseteq\partial_{t,\rho}\mathcal{M}-\mathcal{M}'.
\]
Let $A:\dom(A)\subseteq H\to H$ be skew-selfadjoint. Furthermore,
assume there exists $c>0$ such that
\[
\Re\langle(\partial_{t,\rho}\mathcal{M}+\mathcal{N})u,u\rangle_{\rho,0}\geq c\|u\|_{\rho,0}^{2}\quad(u\in H_{\rho}^{1}(\R;H)).
\]
Then the operator $\partial_{t,\rho}\mathcal{M}+\mathcal{N}+A$ is
closable, and its closure is continuously invertible. Here, $A$ is
identified with its canonical extension to a skew-selfadjoint operator
on $L_{2,\rho}(\R;H)$ with domain $L_{2,\rho}(\R;\dom(A)).$ 
\end{thm}

\begin{rem}
\begin{enumerate}[(a)]

\item In \cite[Theorem 3.4]{Waurick2015} the assumptions in the
operators involved are a little bit weaker, but for our purposes,
this version of the theorem is sufficient; for a comprehensive discussion,
see \cite[Theorems 3.3.2]{W16_H} for the version above and \cite[Theorems 3.4.6]{W16_H}
for the corresponding variant in \cite{Waurick2015}.

\item \prettyref{thm:sol_theory_basic} provides a unified solution
theory for a broad class of non-autonomous problems. Due to the flexibility
of the choice of the operators $\mathcal{M}$ and $\mathcal{N},$
which act in space-time, the problem class comprises many different
types of differential equations, like delay equations, fractional
differential equations, integro-differential equations and coupled
problems thereof (see e.g. \cite{PTW2015_survey,seifert2020evolutionary}
for some survey in the autonomous case and \cite{W16_H,Trostorff2020a,PTWW13_NA}
for some non-autonomous and/or nonlinear examples). 

\end{enumerate}
\end{rem}

\section{The solution theory in $H_{\rho}^{1/2}(\R;H)$\label{sec:The-solution-theory H12}}

In this section, we have a closer look at the solution theory for
evolutionary equations in $H_{\rho}^{1/2}(\R;H)$; that is, we prove
an analogous statement to \prettyref{thm:sol_theory_basic}, but now
the equation is considered as an equation on $H_{\rho}^{1/2}(\R;H).$
The basic setting is the following. Let $\mathcal{M},\mathcal{N},\mathcal{M}'\in L(L_{2,\rho}(\R;H))$
with the following properties:
\[
\Re\langle\left(\partial_{t,\rho}\mathcal{M}+\mathcal{N}\right)\phi,\phi\rangle_{\rho,1/2}\geq c\langle\phi,\phi\rangle_{\rho,1/2}
\]
for all $\phi\in H_{\rho}^{3/2}(\R;H)$ some $\rho>0$ and $c>0$.
Moreover, we assume that 
\[
\mathcal{M}\partial_{t,\rho}\subseteq\partial_{t,\rho}\mathcal{M}-\mathcal{M}'
\]
and $\mathcal{N}|_{H_{\rho}^{1/2}(\R;H)},\mathcal{M}'|_{H_{\rho}^{1/2}(\R;H)}\in L(H_{\rho}^{1/2}(\R;H))$.

We discuss operators on $H_{\rho}^{1/2}(\R;H)$ in more detail next.
For this we recall the notation of the commutator of two operators
$S$ and $T$ on some Hilbert space $H$,
\[
[S,T]\coloneqq ST-TS,\quad\text{ with }\dom([S,T])=\dom(ST)\cap\dom(TS).
\]
In the case that $[S,T]$ is densely defined in $H$ and extends to
a bounded linear operator on $H$, we omit the closure bar and just
write $[S,T]\in L(H)$. Consequently, we also use $[S,T]$ to denote
the (then continuous operator) $\overline{[S,T]}$.
\begin{lem}
\label{lem:12comm}

(a) Let $\mathcal{C}\in L(L_{2,\rho}(\R;H)).$ Then $\mathcal{C}|_{H_{\rho}^{1/2}}\in L(H_{\rho}^{1/2}(\R;H))$
if and only if $\partial_{t,\rho}^{1/2}\mathcal{\mathcal{C}}\partial_{t,\rho}^{-1/2}\in L(L_{2,\rho}(\R;H))$
in either case, we have 
\[
\|\mathcal{C}|_{H_{\rho}^{1/2}}\|_{L(H_{\rho}^{1/2}(\R;H))}=\|\partial_{t,\rho}^{1/2}\mathcal{C}|_{H_{\rho}^{1/2}}\partial_{t,\rho}^{-1/2}\|_{L(L_{2,\rho}(\R;H))}.
\]
Either of the alternative conditions is particularly satisfied if
$[\mathcal{C},\partial_{t,\rho}^{1/2}]\in L(L_{2,\rho}(\R;H))$. Moreover,
in this case we have
\begin{align*}
\|\mathcal{C}|_{H_{\rho}^{1/2}}\|_{L(H_{\rho}^{1/2}(\R;H))} & \leq\|\mathcal{C}\|_{L(L_{2,\rho}(\R;H))}+\frac{1}{\sqrt{\rho}}\|[\mathcal{C},\partial_{t,\rho}^{1/2}]\|_{L(L_{2,\rho}(\R;H))}.
\end{align*}

(b) $[\mathcal{N},(1+\varepsilon\partial_{t,\rho})^{-1}]\to0$ strongly
in $L(H_{\rho}^{1/2}(\R;H))$ as $\varepsilon\to0+$. 

(c) $\left[\partial_{t,\rho}\mathcal{M},(1+\varepsilon\partial_{t,\rho})^{-1}\right]\in L(H_{\rho}^{1/2}(\R;H))$
and $\left[\partial_{t,\rho}\mathcal{M},(1+\varepsilon\partial_{t,\rho})^{-1}\right]\to0$
strongly in $L(H_{\rho}^{1/2}(\R;H))$ as $\varepsilon\to0+$.
\end{lem}

\begin{proof}
(a) Let $\phi\in C_{c}^{\infty}(\R;H)$. Assume that $\partial_{t,\rho}^{1/2}\mathcal{\mathcal{C}}\partial_{t,\rho}^{-1/2}\in L(L_{2,\rho}(\R;H)).$
Then we compute
\begin{align*}
\|\mathcal{C}\phi\|_{\rho,\frac{1}{2}} & =\|\partial_{t,\rho}^{1/2}\mathcal{C}\partial_{t,\rho}^{-1/2}\partial_{t,\rho}^{1/2}\phi\|_{\rho,0}\leq\|\partial_{t,\rho}^{1/2}\mathcal{C}\partial_{t,\rho}^{-1/2}\|_{L(L_{2,\rho}(\R;H))}\|\phi\|_{\rho,\frac{1}{2}}.
\end{align*}
If, on the other hand, $\mathcal{C}|_{H_{\rho}^{1/2}}\in L(H_{\rho}^{1/2}(\R;H))$.
Then $\mathcal{\partial}_{t,\rho}^{1/2}\mathcal{C}|_{H_{\rho}^{1/2}}\partial_{t,\rho}^{-1/2}\in L(L_{2,\rho}(\R;H))$
since $\partial_{t,\rho}^{-1/2}\in L(L_{2,\rho}(\R;H),H_{\rho}^{1/2}(\R;H))$
and $\partial_{t,\rho}^{1/2}\in L(H_{\rho}^{1/2}(\R;H),L_{2,\rho}(\R;H))$
are unitary.\\
Assume now that $[\mathcal{C},\partial_{t,\rho}^{1/2}]\in L(L_{2,\rho}(\R;H))$.
Then 
\[
\partial_{t,\rho}^{1/2}\mathcal{C}\partial_{t,\rho}^{-1/2}=[\partial_{t,\rho}^{1/2},\mathcal{C}]\partial_{t,\rho}^{-1/2}+\mathcal{C}\in L(L_{2,\rho}(\R;H))
\]
and, using \prettyref{prop:fractional_integral}, we get
\begin{align*}
\|\mathcal{C}|_{H_{\rho}^{1/2}}\|_{L(H_{\rho}^{1/2}(\R;H))} & =\|\partial_{t,\rho}^{1/2}\mathcal{C}|_{H_{\rho}^{1/2}}\partial_{t,\rho}^{-1/2}\|_{L(L_{2,\rho}(\R;H))}\\
 & \leq\|\mathcal{C}\|_{L(L_{2,\rho}(\R;H)}+\|[\mathcal{C},\partial_{t,\rho}^{1/2}]\partial_{t,\rho}^{-1/2}\|_{L(L_{2,\rho}(\R;H))}\\
 & \leq\|\mathcal{C}\|_{L(L_{2,\rho}(\R;H)}+\frac{1}{\sqrt{\rho}}\|[\mathcal{C},\partial_{t,\rho}^{1/2}]\|_{L(L_{2,\rho}(\R;H))}.
\end{align*}

(b) Let $\varepsilon>0$. By (a) together with the proved inequality,
it suffices to show that $\partial_{t,\rho}^{1/2}[\mathcal{N},(1+\varepsilon\partial_{t,\rho})^{-1}]\partial_{t,\rho}^{-1/2}\to0$
strongly in $L(L_{2,\rho}(\R;H))$. For this, we compute 
\begin{align*}
\partial_{t,\rho}^{1/2}[\mathcal{N},(1+\varepsilon\partial_{t,\rho})^{-1}]\partial_{t,\rho}^{-1/2} & =\partial_{t,\rho}^{1/2}\mathcal{N}(1+\varepsilon\partial_{t,\rho})^{-1}\partial_{t,\rho}^{-1/2}-\partial_{t,\rho}^{1/2}(1+\varepsilon\partial_{t,\rho})^{-1}\mathcal{N}\partial_{t,\rho}^{-1/2}\\
 & =\left[\partial_{t,\rho}^{1/2}\mathcal{N}\partial_{t,\rho}^{-1/2},(1+\varepsilon\partial_{t,\rho})^{-1}\right],
\end{align*}
the latter tends to $0$ since $\partial_{t,\rho}^{1/2}\mathcal{N}\partial_{t,\rho}^{-1/2}\in L(L_{2,\rho}(\R;H))$,
by part (a) and $(1+\varepsilon\partial_{t,\rho})^{-1}\to1$ strongly
as $\varepsilon\to0$ by \prettyref{lem:approx}.

(c) Let $\varepsilon>0$. Then we compute using $\mathcal{M}'=\left[\partial_{t,\rho},\mathcal{M}\right]$
\begin{align*}
\left[\partial_{t,\rho}\mathcal{M},(1+\varepsilon\partial_{t,\rho})^{-1}\right] & =\left(\partial_{t,\rho}\mathcal{M}(1+\varepsilon\partial_{t,\rho})^{-1}-(1+\varepsilon\partial_{t,\rho})^{-1}\partial_{t,\rho}\mathcal{M}\right)\\
 & =\partial_{t,\rho}\left(\mathcal{M}(1+\varepsilon\partial_{t,\rho})^{-1}-(1+\varepsilon\partial_{t,\rho})^{-1}\mathcal{M}\right)\\
 & =\partial_{t,\rho}(1+\varepsilon\partial_{t,\rho})^{-1}\left((1+\varepsilon\partial_{t,\rho})\mathcal{M}-\mathcal{M}(1+\varepsilon\partial_{t,\rho})\right)(1+\varepsilon\partial_{t,\rho})^{-1}\\
 & =\varepsilon\partial_{t,\rho}(1+\varepsilon\partial_{t,\rho})^{-1}\mathcal{M}'(1+\varepsilon\partial_{t,\rho})^{-1}.
\end{align*}
Hence, we deduce
\begin{align*}
 & \partial_{t,\rho}^{1/2}\left[\partial_{t,\rho}\mathcal{M},(1+\varepsilon\partial_{t,\rho})^{-1}\right]\partial_{t,\rho}^{-1/2}\\
 & =\varepsilon\partial_{t,\rho}(1+\varepsilon\partial_{t,\rho})^{-1}\partial_{t,\rho}^{1/2}\mathcal{M}'\partial_{t,\rho}^{-1/2}(1+\varepsilon\partial_{t,\rho})^{-1}\in L(L_{2,\rho}(\R;H))
\end{align*}
and thus, $\left[\partial_{t,\rho}\mathcal{M},(1+\varepsilon\partial_{t,\rho})^{-1}\right]\in L(H_{\rho}^{1/2}(\R;H))$
by (a). Moreover, by \prettyref{lem:approx}
\begin{align*}
\partial_{t,\rho}^{1/2}\left[\partial_{t,\rho}\mathcal{M},(1+\varepsilon\partial_{t,\rho})^{-1}\right]\partial_{t,\rho}^{-1/2} & =\varepsilon\partial_{t,\rho}(1+\varepsilon\partial_{t,\rho})^{-1}\partial_{t,\rho}^{1/2}\mathcal{M}'\partial_{t,\rho}^{-1/2}(1+\varepsilon\partial_{t,\rho})^{-1}\\
 & =(1-(1+\varepsilon\partial_{t,\rho})^{-1})\partial_{t,\rho}^{1/2}\mathcal{M}'\partial_{t,\rho}^{-1/2}(1+\varepsilon\partial_{t,\rho})^{-1}\\
 & \to0\quad(\varepsilon\to0)
\end{align*}
strongly in $L(L_{2,\rho}(\R;H))$, which yields the asserted convergence
again by part (a).
\end{proof}
\begin{prop}
\label{prop:range_d0M0}We have $\partial_{t,\rho}\mathcal{M}[H_{\rho}^{3/2}(\R;H)]\subseteq H_{\rho}^{1/2}(\R;H)$. 
\end{prop}

\begin{proof}
Let $\phi\in\dom(\partial_{t,\rho}^{3/2})$. Then
\begin{align*}
\partial_{t,\rho}\mathcal{M}\phi & =\mathcal{M}'\phi+\mathcal{M}\partial_{t,\rho}\phi.
\end{align*}
Since $\mathcal{M}'\in L(H_{\rho}^{1/2}(\R;H))$, we obtain $\mathcal{M}'\phi\in H_{\rho}^{1/2}(\R;H).$
Furthermore, since $\mathcal{M}\in L(L_{2,\rho}(\R;H))\cap L(H_{\rho}^{1}(\R;H)),$
we deduce $\mathcal{M}\in L(H_{\rho}^{1/2}(\R;H))$ by complex interpolation
(see \prettyref{rem:interpolation} (b)). Hence, we also have $\mathcal{M}\partial_{t,\rho}\phi\in H_{\rho}^{1/2}(\R;H),$
which shows the assertion. 
\end{proof}
\begin{lem}
\label{lem:core_M}For $\varepsilon>0$ we set $R_{\varepsilon}\coloneqq(1+\varepsilon\partial_{t,\rho})^{-1}.$
Let $u\in\dom(\partial_{t,\rho}\mathcal{M})\subseteq H_{\rho}^{1/2}(\R;H).$
Then for each $\varepsilon>0$ and $k\in\N$ we have $R_{\varepsilon}^{k}u\in\dom(\partial_{t,\rho}\mathcal{M})$
and 
\[
\partial_{t,\rho}\mathcal{M}R_{\varepsilon}^{k}u\to\partial_{t,\rho}\mathcal{M}u\quad(\varepsilon\to0+).
\]
In particular $H_{\rho}^{k+1/2}(\R;H)$ is a core for $\partial_{t,\rho}\mathcal{M}$
for each $k\in\N$. 
\end{lem}

\begin{proof}
The proof follows by induction on $k$. For $k=0$ there is nothing
to show. Assume now that the assertion holds for $k-1$. Then we compute,
using \prettyref{lem:12comm} (c) and \prettyref{lem:approx}
\[
\partial_{t,\rho}\mathcal{M}R_{\varepsilon}^{k}u=[\partial_{t,\rho}\mathcal{M},R_{\varepsilon}]R_{\varepsilon}^{k-1}u+R_{\varepsilon}\partial_{t,\rho}\mathcal{M}R_{\varepsilon}^{k-1}u\to\partial_{t,\rho}\mathcal{M}u\quad(\varepsilon\to0+).\tag*{\qedhere}
\]
 
\end{proof}
\begin{thm}
\label{thm:sol_theory_1/2}The operator $\partial_{t,\rho}\mathcal{M}+\mathcal{N}+A$
considered as an operator on $H_{\rho}^{1/2}(\R;H)$ is closable and
its closure is continuously invertible.
\end{thm}

\begin{proof}
Recall that all operators are now considered as operators acting on
$H_{\rho}^{1/2}(\R;H).$ Since $\partial_{t,\rho}\mathcal{M}+\mathcal{N}$
is strictly positive definite in the Hilbert space $H_{\rho}^{1/2}(\R;H)$
with domain $H_{\rho}^{3/2}(\R;H)$ and $A$ is skew-selfadjoint,
we derive that $\partial_{t,\rho}\mathcal{M}+\mathcal{N}+A$ is strictly
positive definite in the Hilbert space $H_{\rho}^{1/2}(\R;H)$ with
domain $H_{\rho}^{3/2}(\R;\dom(A)).$ By \prettyref{lem:core_M} this
positive definiteness extends to all elements in $\dom(\partial_{t,\rho}\mathcal{M}+\mathcal{N}+A)$
and thus, $\partial_{t,\rho}\mathcal{M}+\mathcal{N}+A$ is one-to-one
and has a continuous inverse defined on the range of $\partial_{t,\rho}\mathcal{M}+\mathcal{N}+A$.
Since $H_{\rho}^{3/2}(\R;\dom(A))$ is dense in $H_{\rho}^{1/2}(\R;H),$
the latter implies that $\partial_{t,\rho}\mathcal{M}+\mathcal{N}+A$
is closable (see e.g.~\cite[Proposition 2.3.14]{W16_H} or \cite[Theorem 4.2.5]{Beyer2007}).
Moreover, it is a standard argument to show that $\overline{\partial_{t,\rho}\mathcal{M}+\mathcal{N}+A}$
is continuously invertible on its range, which is closed. Hence, for
showing that $\overline{\partial_{t,\rho}\mathcal{M}+\mathcal{N}+A}$
is onto, it suffices to compute the adjoint and confirm that this
adjoint is one-to-one, which in turn would imply the density of the
range of $\partial_{t,\rho}\mathcal{M}+\mathcal{N}+A$. For doing
so, let $\varepsilon>0$, $u\in\dom\left(\partial_{t,\rho}\mathcal{M}+\mathcal{N}+A\right)$
and $f\in\dom\left(\left(\partial_{t,\rho}\mathcal{M}+\mathcal{N}+A\right)^{*}\right)$.
We put $f_{\varepsilon}\coloneqq(1+\varepsilon\partial_{t,\rho}^{*})^{-1}f\in\dom(\partial_{t,\rho}^{3/2})$.
Then we compute
\begin{align*}
 & \langle\left(\partial_{t,\rho}\mathcal{M}+\mathcal{N}+A\right)u,f_{\varepsilon}\rangle_{\rho,1/2}\\
 & =\langle(1+\varepsilon\partial_{t,\rho})^{-1}\left(\partial_{t,\rho}\mathcal{M}+\mathcal{N}+A\right)u,f\rangle_{\rho,1/2}\\
 & =\langle\left(\partial_{t,\rho}\mathcal{M}+\mathcal{N}+A\right)(1+\varepsilon\partial_{t,\rho})^{-1}u-\\
 & \quad\quad-\left[\partial_{t,\rho}\mathcal{M},(1+\varepsilon\partial_{t,\rho})^{-1}\right]u-\left[\mathcal{N},(1+\varepsilon\partial_{t,\rho})^{-1}\right]u,f\rangle_{\rho,1/2}.
\end{align*}
Thus, invoking \prettyref{lem:12comm} (c), we obtain that 
\[
f_{\varepsilon}\in\dom\left(\partial_{t,\rho}\mathcal{M}+\mathcal{N}+A\right)^{*}
\]
and
\begin{align*}
 & \left(\partial_{t,\rho}\mathcal{M}+\mathcal{N}+A\right)^{*}f_{\varepsilon}\\
 & =(1+\varepsilon\partial_{t,\rho}^{*})^{-1}\left(\partial_{t,\rho}\mathcal{M}+\mathcal{N}+A\right)^{*}f-\left[\partial_{t,\rho}\mathcal{M},(1+\varepsilon\partial_{t,\rho})^{-1}\right]^{*}f-\left[\mathcal{N},(1+\varepsilon\partial_{t,\rho})^{-1}\right]^{*}f.
\end{align*}
Since $\left[\partial_{t,\rho}\mathcal{M},(1+\varepsilon\partial_{t,\rho})^{-1}\right]^{*}f+\left[\mathcal{N},(1+\varepsilon\partial_{t,\rho})^{-1}\right]^{*}f\to0$
weakly in $H_{\rho}^{1/2}(\R;H)$ as $\varepsilon\to0+$ by \prettyref{lem:12comm}
(b) and (c), we infer that $H_{\rho}^{3/2}(\R;H)\cap\dom\left((\partial_{t,\rho}\mathcal{M}+\mathcal{N}+A)^{\ast}\right)$
is a core for $\left(\partial_{t,\rho}\mathcal{M}+\mathcal{N}+A\right)^{*}$
(note that $f_{\varepsilon}\to f$ in $H_{\rho}^{1/2}(\R;H)$ by \prettyref{lem:approx}). 

Next, we prove that $H_{\rho}^{3/2}(\R;H)\cap\dom\left((\partial_{t,\rho}\mathcal{M}+\mathcal{N}+A)^{\ast}\right)\subseteq H_{\rho}^{1/2}(\R;\dom(A)).$
For doing so, let $f\in H_{\rho}^{3/2}(\R;H)\cap\dom\left((\partial_{t,\rho}\mathcal{M}+\mathcal{N}+A)^{\ast}\right)$
and $\psi\in H_{\rho}^{3/2}(\R;\dom(A)).$ Then $\psi\in\dom(\partial_{t,\rho}\mathcal{M})$
by \prettyref{prop:range_d0M0} and we compute 
\begin{align*}
\langle A\psi,f\rangle_{\rho,1/2} & =\langle(\partial_{t,\rho}\mathcal{M}+\mathcal{N}+A)\psi,f\rangle_{\rho,1/2}-\langle(\partial_{t,\rho}\mathcal{M}+\mathcal{N})\psi,f\rangle_{\rho,1/2}\\
 & =\langle\psi,(\partial_{t,\rho}\mathcal{M}+\mathcal{N}+A)^{\ast}f\rangle_{\rho,1/2}-\langle\psi,(\mathcal{M}^{\ast}\partial_{t,\rho}^{\ast}+\mathcal{N}^{\ast})f\rangle_{\rho,1/2},
\end{align*}
where we have used $\mathcal{M}\in L(H_{\rho}^{1/2}(\R;H)),$ see
\prettyref{rem:interpolation} (b). Since $H_{\rho}^{3/2}(\R;\dom(A))$
is dense in $H_{\rho}^{1/2}(\R;\dom(A))$, we infer that $H_{\rho}^{3/2}(\R;\dom(A))$
is a core for $A$. Thus, $f\in H_{\rho}^{1/2}(\R;\dom(A))$ with
\[
-Af=A^{\ast}f=(\partial_{t,\rho}\mathcal{M}+\mathcal{N}+A)^{\ast}f-(\mathcal{M}^{\ast}\partial_{t,\rho}^{\ast}+\mathcal{N}^{\ast})f.
\]
Hence, for $f\in H_{\rho}^{3/2}(\R;H)\cap\dom\left((\partial_{t,\rho}\mathcal{M}+\mathcal{N}+A)^{\ast}\right)$
we compute
\begin{align*}
\Re\langle\left(\partial_{t,\rho}\mathcal{M}+\mathcal{N}+A\right)^{*}f,f\rangle_{\rho,\frac{1}{2}} & =\Re\langle\mathcal{M}^{*}\partial_{t,\rho}^{*}f+\mathcal{N}^{*}f-Af,f\rangle_{\rho,\frac{1}{2}}\\
 & =\Re\langle f,\left(\partial_{t,\rho}\mathcal{M}+\mathcal{N}\right)f\rangle_{\rho,\frac{1}{2}} \geq c\langle f,f\rangle_{\rho,\frac{1}{2}}
\end{align*}
and since the set $H_{\rho}^{3/2}(\R;H)\cap\dom\left((\partial_{t,\rho}\mathcal{M}+\mathcal{N}+A)^{\ast}\right)$
is a core for $(\partial_{t,\rho}\mathcal{M}+\mathcal{N}+A)^{\ast}$,
the latter particularily implies that $\left(\partial_{t,\rho}\mathcal{M}+\mathcal{N}+A\right)^{*}$
is one-to-one. 
\end{proof}

\begin{cor}
\label{cor:density12}For $k\geq1$, the set
\[
\left(\partial_{t,\rho}\mathcal{M}+\mathcal{N}+A\right)[H_{\rho}^{k+1/2}(\R;\dom(A))]
\]
is dense in $H_{\rho}^{1/2}(\R;H)$. 
\end{cor}

\begin{proof}
Let $f\in H_{\rho}^{1/2}(\R;H)$ and set 
\[
u\coloneqq\left(\overline{\partial_{t,\rho}\mathcal{M}+\mathcal{N}+A}\right)^{-1}f\in H_{\rho}^{1/2}(\R;H).
\]
Note that such a $u$ exists by \prettyref{thm:sol_theory_1/2}. Hence,
we find a sequence $(u_{n})_{n\in\N}$ in $\dom(\partial_{t,\rho}\mathcal{M})\cap\dom(A)$
such that 
\[
u_{n}\to u\text{ and }\left(\partial_{t,\rho}\mathcal{M}+\mathcal{N}+A\right)u_{n}\to f\text{ as }n\to\infty
\]
in $H_{\rho}^{1/2}(\R;H).$ We now define $v_{n,\varepsilon}\coloneqq(1+\varepsilon\partial_{t,\rho})^{-k}u_{n}\in H_{\rho}^{k+1/2}(\R;\dom(A))$
for $n\in\N.$ By \prettyref{lem:approx}, $Av_{n,\varepsilon}\to Au_{n}$
as well as $\mathcal{N}v_{n,\varepsilon}\to\mathcal{N}u_{n}$ as $\varepsilon\to0$
and thus, it suffices to show 
\[
\partial_{t,\rho}\mathcal{M}v_{n,\varepsilon}\to\partial_{t,\rho}\mathcal{M}u_{n}\text{ as }\varepsilon\to0.
\]
This, however, follows from \prettyref{lem:core_M} and thus, the
assertion follows.
\end{proof}
\begin{cor}
\label{cor:densityresult}Let $H=H_{0}\oplus H_{1}$ for Hilbert spaces
$H_{0}$ and $H_{1}.$ Then for $k\geq1$, the set
\[
\left(\partial_{t,\rho}\mathcal{M}+\mathcal{N}+A\right)[H_{\rho}^{k+1/2}(\R;\dom(A))]
\]
is dense in $L_{2,\rho}(\R;H_{0})\oplus H_{\rho}^{1/2}(\R;H_{1}).$
\end{cor}

\begin{proof}
By \prettyref{cor:density12}, the set 
\[
\left(\partial_{t,\rho}\mathcal{M}+\mathcal{N}+A\right)[H_{\rho}^{k+1/2}(\R;\dom(A))]
\]
is dense in $H_{\rho}^{1/2}(\R;H)$, which continuously and densely
embeds into $L_{2,\rho}(\R;H_{0})\oplus H_{\rho}^{1/2}(\R;H_{1})$.
This shows the assertion. 
\end{proof}

\section{Maximal regularity for evolutionary equations\label{sec:Maximal-regularity-forEE}}

In the following we provide our main result: a criterion for maximal
regularity for evolutionary equations. In a nutshell this criterion
reads:
\begin{center}
\emph{Well-posedness in both $L_{2,\rho}(\R;H)$ and $H_{\rho}^{1/2}(\R;H)$
together with a parabolic-like structure implies maximal regularity.}
\par\end{center}

Throughout, let $H_{0}$ and $H_{1}$ be two complex Hilbert spaces
and set $H\coloneqq H_{0}\oplus H_{1}.$ Moreover, let $C:\dom(C)\subseteq H_{0}\to H_{1}$
densely defined closed and linear and set 
\begin{equation}
A=\left(\begin{array}{cc}
0 & -C^{*}\\
C & 0
\end{array}\right)\label{eq:A is C star C}
\end{equation}
(which easily can be verified to be skew-selfadjoint in $H$). Finally,
we assume that $\mathcal{M}$ and $\mathcal{N}$ have the form 
\begin{equation}
\mathcal{M}=\left(\begin{array}{cc}
\mathcal{M}_{00} & 0\\
0 & 0
\end{array}\right)\label{eq:M is M00}
\end{equation}
 as well as 
\begin{equation}
\mathcal{N}=\left(\begin{array}{cc}
\mathcal{N}_{00} & \mathcal{N}_{01}\\
\mathcal{N}_{10} & \mathcal{N}_{11}
\end{array}\right)\label{eq:N is Matrix}
\end{equation}
 with appropriate linear operators in $L(L_{2,\rho}(\R;H_{j}),L_{2,\rho}(\R;H_{i}))$,
$i,j\in\{0,1\}$. 
\begin{thm}
\label{thm:main}Let $A,\mathcal{M}$ and $\mathcal{N}$ be as in
\prettyref{eq:A is C star C}, \prettyref{eq:M is M00}, and \prettyref{eq:N is Matrix}
and assume there is $\mathcal{M}'\in L(L_{2,\rho}(\R;H))$ with 
\begin{equation}
\mathcal{M}\partial_{t,\rho}\subseteq\partial_{t,\rho}\mathcal{M}-\mathcal{M}'.\label{eq:commM}
\end{equation}
Assume, in addition, that 
\[
\mathcal{M}',\mathcal{N}\in L(H_{\rho}^{1/2}(\R;H)).
\]
We shall assume the positive definiteness conditions
\begin{equation}
\Re\langle\mathcal{M}_{00}\phi,\phi\rangle_{\rho,0}\geq c\langle\phi,\phi\rangle_{\rho,0},\label{eq:M00pd}
\end{equation}
\begin{equation}
\Re\langle\left(\partial_{t,\rho}\mathcal{M}+\mathcal{N}\right)\phi,\phi\rangle_{\rho,0}\geq c\langle\phi,\phi\rangle_{\rho,0},\label{eq:pd0}
\end{equation}
and
\begin{equation}
\Re\langle\left(\partial_{t,\rho}\mathcal{M}+\mathcal{N}\right)\phi,\phi\rangle_{\rho,\frac{1}{2}}\geq c\langle\phi,\phi\rangle_{\rho,\frac{1}{2}}\label{eq:pd12}
\end{equation}
for some $c>0$ and all $\phi\in H_{\rho}^{3/2}(\R;H).$ Let $\mathcal{S}_{\rho}\coloneqq\left(\overline{\partial_{t,\rho}\mathcal{M}+\mathcal{N}+A}\right)^{-1}\in L(L_{2,\rho}(\R;H))\cap L(H_{\rho}^{1/2}(\R;H))$
(cp. \prettyref{thm:sol_theory_basic} and \prettyref{thm:sol_theory_1/2}).
Then 
\[
\mathcal{S}_{\rho}[L_{2,\rho}(\R;H_{0})\times H_{\rho}^{1/2}(\R;H_{1})]\subseteq\left(H_{\rho}^{1}(\R;H_{0})\cap L_{2,\rho}(\R;\dom(C))\right)\times L_{2,\rho}(\R;\dom(C^{\ast}));
\]
that is, for $f\in L_{2,\rho}(\R;H_{0})$ and $g\in H_{\rho}^{1/2}(\R;H_{1})$
and $(u,v)\in L_{2,\rho}(\R;H)$ satisfying 
\[
\overline{\left(\partial_{t,\rho}\mathcal{M}+\mathcal{N}+A\right)}\left(\begin{array}{c}
u\\
v
\end{array}\right)=\left(\begin{array}{c}
f\\
g
\end{array}\right),
\]
we have
\[
u\in H_{\rho}^{1}(\R;H_{0})\cap L_{2,\rho}(\R;\dom(C))\text{ and }v\in L_{2,\rho}(\R;\dom(C^{\ast})).
\]
Moreover, we have
\[
u\in H_{\rho}^{1/2}(\R;\dom(C)).
\]
\end{thm}

\begin{rem}
As we shall see in the examples section, the above nutshell description
of the \prettyref{thm:main} is visible as follows: \begin{itemize}

\item Well-posedness in $L_{2,\rho}(\mathbb{R};H)$ is guaranteed
by assumption \prettyref{eq:pd0}; see \prettyref{thm:sol_theory_basic}.

\item Well-posedness in $H_{\rho}^{1/2}(\mathbb{R};H)$ is guaranteed
by assumption \prettyref{eq:pd12}; see \prettyref{thm:sol_theory_1/2}.

\item The parabolic-like structure is visible in the block matrix
structure \prettyref{eq:A is C star C}, \prettyref{eq:M is M00}
and the positive definiteness condition \prettyref{eq:M00pd}.

\end{itemize}
\end{rem}

\begin{rem}
\label{rem:regPhen}(a) As in \cite{Auscher_Egert2016,Dier_Zacher2017},
we recover the same additional regularity phenomenon $u\in H_{\rho}^{1/2}(\mathbb{R};\dom(C)),$
which is not expected for maximal regularity of evolutionary equations.
In fact, as the proof of \prettyref{thm:main} will show an estimate
for $\|Cu\|_{\rho,\frac{1}{2}}$ is key for obtaining the main result.

(b) Note that $u\in H_{\rho}^{1/2}(\mathbb{R};\dom(C))$ also has
a consequence on the time-regularity of $v$. Indeed, taking $(f,g)$
as right-hand sides as in \prettyref{thm:main}, we see that $(u,v)$
satisfies
\begin{align*}
\partial_{t,\rho}\mathcal{M}_{00}u+\mathcal{N}_{00}u+\mathcal{N}_{01}v-C^{*}v & =f\\
\mathcal{N}_{10}u+\mathcal{N}_{11}v+Cu & =g.
\end{align*}
Focussing on the second line and multiplying by $\mathcal{N}_{11}^{-1}$,
we infer 
\[
v=\mathcal{N}_{11}^{-1}\left(g-Cu-\mathcal{N}_{11}u\right).
\]
Next, since $u\in H_{\rho}^{1}(\mathbb{R};H_{0})\subseteq H_{\rho}^{1/2}(\mathbb{R};H_{0})$,
$Cu,g\in H_{\rho}^{1/2}(\mathbb{R};H_{1})$ and both $\mathcal{N}_{01}$
and $\mathcal{N}_{11}^{-1}$ leave $H_{\rho}^{1/2}$-regular mappings
invariant (see also \prettyref{lem:N11} below), we infer $v\in H_{\rho}^{1/2}(\mathbb{R};H_{1})$.
We summarise all the regularity results in the next statement.
\end{rem}

\begin{cor}
\label{cor:main_result1}Under the assumptions of \prettyref{thm:main},
let $f\in L_{2,\rho}(\R;H_{0})$, $g\in H_{\rho}^{1/2}(\R;H_{1})$
and $(u,v)\in L_{2,\rho}(\R;H)$ satisfying 
\[
\overline{\left(\partial_{t,\rho}\mathcal{M}+\mathcal{N}+A\right)}\left(\begin{array}{c}
u\\
v
\end{array}\right)=\left(\begin{array}{c}
f\\
g
\end{array}\right).
\]
Then
\begin{align*}
u & \in H_{\rho}^{1}(\mathbb{R};H_{0})\cap H_{\rho}^{1/2}(\mathbb{R};\dom(C))\text{ and }\\
v & \in L_{2,\rho}(\mathbb{R};\dom(C^{*}))\cap H_{\rho}^{1/2}(\mathbb{R};H_{1}).
\end{align*}
\end{cor}

\begin{rem}
\label{rem:(a)contestim(b)2ndorder}(a) Note that the regularity statement
in the latter result is also accompanied with the corresponding continuity
statement; that is, there exists a constant $\kappa\geq0$ such that
for all $f\in L_{2,\rho}(\R;H_{0}),g\in H_{\rho}^{1/2}(\R;H_{1})$
with $\mathcal{S}_{\rho}(f,g)=(u,v)$ we have
\[
\|u\|_{\rho,1}+\|Cu\|_{\rho,\frac{1}{2}}+\|C^{*}v\|_{\rho,0}+\|v\|_{\rho,\frac{1}{2}}\leq\kappa\left(\|f\|_{\rho,0}+\|g\|_{\rho,\frac{1}{2}}\right).
\]
Moreover, note that, as a consequence, the closure bar in the formulation
of the evolutionary equation can be omitted so that, indeed,
\[
\overline{\left(\partial_{t,\rho}\mathcal{M}+\mathcal{N}+A\right)}^{-1}\left(\begin{array}{c}
f\\
g
\end{array}\right)=\left(\partial_{t,\rho}\mathcal{M}+\mathcal{N}+A\right)^{-1}\left(\begin{array}{c}
f\\
g
\end{array}\right),
\]
addressing \prettyref{prob:LionsEE}.

(b) We note here that \prettyref{cor:main_result1} is in fact a result
on maximal regularity for the special case of $g=0$ and if the evolutionary
equation is viewed as a `scalar' equation in the following sense.
Assume $g=0$ and $f\in L_{2,\rho}(\mathbb{R};H_{0})$, then as in
\prettyref{rem:regPhen}(b), we have seen that then
\begin{align*}
\partial_{t,\rho}\mathcal{M}_{00}u+\mathcal{N}_{00}u+\mathcal{N}_{01}v-C^{*}v & =f\\
\mathcal{N}_{10}u+\mathcal{N}_{11}v+Cu & =0.
\end{align*}
Rearranging the second equality, we infer $v=-\mathcal{N}_{11}^{-1}(Cu+\mathcal{N}_{10}u)\in\dom(C^{*})$
and thus the first equation reads 
\begin{equation}
\partial_{t,\rho}\mathcal{M}_{00}u+\mathcal{N}_{00}u-\mathcal{N}_{01}\mathcal{N}_{11}^{-1}(Cu+\mathcal{N}_{10}u)+C^{*}\mathcal{N}_{11}^{-1}(C+\mathcal{N}_{10})u=f.\label{eq:2nd order}
\end{equation}
By \prettyref{eq:M00pd} and \prettyref{eq:commM}, it is not difficult
to see that $\dom(\partial_{t,\rho}\mathcal{M}_{00})=\dom(\partial_{t,\rho})$
and so $u$ as in \prettyref{eq:2nd order} really admits the maximal
regularity hoped for when $f\in L_{2,\rho}(\mathbb{R};H)$.
\end{rem}

In order to prove our main theorem, we need some prerequisites. 
\begin{lem}
\label{lem:N11}Assume the conditions imposed on $\mathcal{M}$ and
$\mathcal{N}$ in \prettyref{thm:main}. Then 
\[
\mathcal{N}_{11}^{-1}\in L(L_{2,\rho}(\R;H))\cap L(H_{\rho}^{1/2}(\R;H)).
\]
Moreover,
\[
\Re\langle\mathcal{N}_{11}^{-1}\phi_{1},\phi_{1}\rangle_{\rho,1/2}\geq\frac{c}{\|\mathcal{N}_{11}\|_{L(H_{\rho}^{1/2}(\R;H_{1}))}^{2}}\langle\phi_{1},\phi_{1}\rangle_{\rho,1/2}\quad(\phi_{1}\in H_{\rho}^{1/2}(\R;H_{1})).
\]
\end{lem}

\begin{proof}
The positive definitness condition in \prettyref{eq:pd0} and \prettyref{eq:pd12}
applied for $\phi=(0,\phi_{1})$ with $\phi_{1}\in H_{\rho}^{3/2}(\R;H_{1})$
implies for $\alpha\in\{0,\frac{1}{2}\}$
\[
\Re\langle\mathcal{N}_{11}\phi_{1},\phi_{1}\rangle_{\rho,\alpha}\geq c\langle\phi_{1},\phi_{1}\rangle_{\rho,\alpha}.
\]
By density, this estimate extends to all $\phi_{1}\in H_{\rho}^{\alpha}(\R;H_{1})$.
Since $\mathcal{N}_{11}\in L(L_{2,\rho}(\R;H))\cap L(H_{\rho}^{1/2}(\R;H)),$
we deduce the first statement. A standard argument also reveals that
(see \cite[Proposition 6.2.3 (b)]{seifert2020evolutionary})
\[
\Re\langle\mathcal{N}_{11}^{-1}\phi_{1},\phi_{1}\rangle_{\rho,\alpha}\geq\frac{c}{\|\mathcal{N}_{11}\|_{L(H_{\rho}^{\alpha}(\R;H_{1}))}^{2}}\langle\phi_{1},\phi_{1}\rangle_{\rho,\alpha}\tag*{\qedhere}.
\]
\end{proof}
\begin{lem}
\label{lem:M00}Assume the conditions imposed on $\mathcal{M}$ and
$\mathcal{N}$ in \prettyref{thm:main}. Then 
\[
\Re\langle\partial_{t,\rho}\mathcal{M}_{00}\phi_{0},|\partial_{t,\rho}|\phi_{0}\rangle_{\rho,0}\geq\left(c-\|\mathcal{N}_{00}\|_{L(H_{\rho}^{1/2}(\R;H_{0}))}\right)\|\phi_{0}\|_{\rho,1/2}^{2}\quad(\phi_{0}\in H_{\rho}^{1}(\R;H_{0})).
\]
\end{lem}

\begin{proof}
We apply the condition \prettyref{eq:pd12} to $\phi=(\phi_{0},0)\in H_{\rho}^{3/2}(\R;H)$
and obtain
\[
\Re\langle\partial_{t,\rho}\mathcal{M}_{00}\phi_{0}+\mathcal{N}_{00}\phi_{0},\phi_{0}\rangle_{\rho,1/2}\geq c\|\phi_{0}\|_{\rho,1/2}^{2}.
\]
Rearranging terms, we obtain
\begin{align*}
\Re\langle\partial_{t,\rho}\mathcal{M}_{00}\phi_{0},|\partial_{t,\rho}|\phi_{0}\rangle_{\rho,0} & =\Re\langle\partial_{t,\rho}\mathcal{M}_{00}\phi_{0},\phi_{0}\rangle_{\rho,1/2}\\
 & \geq c\|\phi_{0}\|_{1/2}^{2}-\Re\langle\mathcal{N}_{00}\phi_{0},\phi_{0}\rangle_{\rho,1/2}\\
 & \geq\left(c-\|\mathcal{N}_{00}\|_{L(H_{\rho}^{1/2}(\R;H))}\right)\|\phi_{0}\|_{\rho,1/2}^{2}.
\end{align*}
Since $H_{\rho}^{3/2}(\R;H_{0})$ is dense in $H_{\rho}^{1}(\mathbb{R};H_{0})$
the assertion follows. 
\end{proof}
\begin{lem}
\label{lem:estimate_root}Let $P:\R\to\R$ be a polyomial of degree
$k\in\N$; that is, $P(x)=\sum_{i=0}^{k}a_{i}x^{i}$ for some $a_{i}\in\R$
with $a_{k}\ne0$. Let $x_{0}\in\R_{\geq0}$ with $x_{0}^{\ell}\leq P(x_{0})$
for some $\ell>k.$ Then 
\[
x_{0}\leq\max\left\{ \sum_{i=0}^{k}|a_{i}|,\left(\sum_{i=0}^{k}|a_{i}|\right)^{\frac{1}{\ell}}\right\} .
\]
\end{lem}

\begin{proof}
Consider the polynomial $Q(x)\coloneqq x^{\ell}-P(x).$ Then $Q(x_{0})\leq0$
and $Q(x)\to\infty$ as $x\to\infty.$ Thus, there exists some $x_{1}\geq x_{0}$
with $Q(x_{1})=0.$ We estimate 
\[
x_{1}^{\ell}=P(x_{1})=\sum_{i=0}^{k}a_{i}x_{1}^{i}\leq\sum_{i=0}^{k}|a_{i}|x_{1}^{i}.
\]
We consider the cases $x_{1}\leq1$ and $x_{1}>1$ separately. Assume
first that $x_{1}\leq1$. Then $x_{1}^{\ell}\leq\sum_{i=0}^{k}|a_{i}|$
and hence, 
\[
x_{0}\leq x_{1}\leq\left(\sum_{i=0}^{k}|a_{i}|\right)^{\frac{1}{\ell}}.
\]
In the case that $x_{1}>1$ we can estimate 
\[
x_{1}^{\ell}\leq\sum_{i=0}^{k}|a_{i}|x_{1}^{i}\leq x_{1}^{\ell-1}\sum_{i=0}^{k}|a_{i}|,
\]
which yields 
\[
x_{0}\leq x_{1}\leq\sum_{i=0}^{k}|a_{i}|.\tag*{\qedhere}
\]
\end{proof}
\begin{proof}[Proof of \prettyref{thm:main}]
 First, let $(f,g)\in(\partial_{t,\rho}\mathcal{M}+\mathcal{N}+A)[H_{\rho}^{2}(\R;\dom(A))]$
and $(u,v)\coloneqq\mathcal{S}_{\rho}(f,g)$. Hence, by the uniqueness
of the solution in $L_{2,\rho}(\mathbb{R};H)$, we obtain 
\[
(u,v)\in H_{\rho}^{2}(\mathbb{R};\dom(A)).
\]
Then we can read off the equations line by line and obtain 
\begin{align*}
\partial_{t,\rho}\mathcal{M}_{00}u+\mathcal{N}_{00}u+\mathcal{N}_{01}v-C^{\ast}v & =f,\\
\mathcal{N}_{10}u+\mathcal{N}_{11}v+Cu & =g.
\end{align*}

Since $\mathcal{N}_{01}$ and $\mathcal{N}_{11}$ are bounded linear
operators mapping $H_{\rho}^{1/2}$ into $H_{\rho}^{1/2}$, we deduce
that the second line particularily implies $g\in H_{\rho}^{1/2}(\R;H_{1}).$
We estimate using \prettyref{eq:M00pd}
\begin{align*}
c\|u\|_{\rho,1}^{2} & \leq\Re\langle\mathcal{M}_{00}\partial_{t,\rho}u,\partial_{t,\rho}u\rangle_{\rho,0}\\
 & =\Re\langle\partial_{t,\rho}\mathcal{M}_{00}u-\mathcal{M}_{00}'u,\partial_{t,\rho}u\rangle_{\rho,0}\\
 & \leq\Re\langle\partial_{t,\rho}\mathcal{M}_{00}u,\partial_{t,\rho}u\rangle_{\rho,0}+\|\mathcal{M}_{00}'\|_{L(L_{2,\rho}(\mathbb{R};H_{0}))}\|u\|_{\rho,0}\|u\|_{\rho,1},
\end{align*}
where $\mathcal{M}_{00}'=[\partial_{t,\rho},\mathcal{M}_{00}]$, which
is bounded in $L(L_{2,\rho}(\mathbb{R};H_{0}))$ by \prettyref{eq:M is M00}
and \prettyref{eq:commM}. Moreover, 
\begin{align*}
 & \Re\langle\partial_{t,\rho}\mathcal{M}_{00}u,\partial_{t,\rho}u\rangle_{\rho,0}\\
 & =\Re\langle f-\mathcal{N}_{00}u-\mathcal{N}_{01}v+C^{\ast}v,\partial_{t,\rho}u\rangle_{\rho,0}\\
 & \leq(\|f\|_{\rho,0}+\|\mathcal{N}_{00}\|_{L(L_{2,\rho}(\mathbb{R};H_{0}))}\|u\|_{\rho,0}+\|\mathcal{N}_{01}\|_{L(L_{2,\rho}(\mathbb{R};H_{0}))}\|v\|_{\rho,0})\|u\|_{\rho,1}\\
 & \quad+\Re\langle C^{\ast}v,\partial_{t,\rho}u\rangle_{\rho,0}.
\end{align*}
The last term can further be estimated by (using \prettyref{lem:N11}
and \prettyref{lem:12comm})
\begin{align*}
\Re\langle C^{\ast}v,\partial_{t,\rho}u\rangle_{\rho,0} & =\Re\langle v,\partial_{t,\rho}Cu\rangle_{\rho,0}\\
 & =\Re\langle\mathcal{N}_{11}^{-1}(g-Cu-\mathcal{N}_{10}u),\partial_{t,\rho}Cu\rangle_{\rho,0}\\
 & =\Re\langle\left(\partial_{t,\rho}^{\ast}\right)^{1/2}\mathcal{N}_{11}^{-1}(g-Cu-\mathcal{N}_{10}u),\partial_{t,\rho}^{1/2}Cu\rangle_{\rho,0}\\
 & =\Re\langle\left(\partial_{t,\rho}^{\ast}\right)^{1/2}\partial_{t,\rho}^{-1/2}\partial_{t,\rho}^{1/2}\mathcal{N}_{11}^{-1}\partial_{t,\rho}^{-1/2}\partial_{t,\rho}^{1/2}(g-Cu-\mathcal{N}_{10}u),\partial_{t,\rho}^{1/2}Cu\rangle_{\rho,0}\\
 & \lesssim\left(\|g\|_{\rho,1/2}+\|u\|_{\rho,1/2}+\|Cu\|_{\rho,1/2}\right)\|Cu\|_{\rho,1/2},
\end{align*}
where $\lesssim$ means an estimate including constants depending
on the operators $\mathcal{N}$ and $\mathcal{M}$ and the positive
definiteness constant $c>0$; we also used $\|\left(\partial_{t,\rho}^{\ast}\right)^{1/2}\partial_{t,\rho}^{-1/2}\|_{L(L_{2,\rho}(\mathbb{R};H))}\leq1$,
which is immediately verified with the help of the Fourier--Laplace
transformation, see \prettyref{prop:fractional_integral}. Thus, summarising
we have shown 
\begin{align*}
\|u\|_{\rho,1}^{2} & \lesssim\left(\|u\|_{\rho,0}+\|f\|_{\rho,0}+\|v\|_{\rho,0}\right)\|u\|_{\rho,1}+\left(\|g\|_{\rho,1/2}+\|u\|_{\rho,1/2}+\|Cu\|_{\rho,1/2}\right)\|Cu\|_{\rho,1/2}.
\end{align*}
Using that $(u,v)=\mathcal{S}_{\rho}(f,g)$ we obtain $\|(u,v)\|_{\rho,0}\lesssim\|(f,g)\|_{\rho,0}\lesssim\|(f,g)\|_{\rho,(0,1/2)},$
where the last norm means that we take the $L_{2,\rho}$ norm of $f$
and the $H_{\rho}^{1/2}$ norm of $g.$ Hence, we can estimate further
\begin{equation}
\|u\|_{\rho,1}^{2}\lesssim\|(f,g)\|_{\rho,(0,1/2)}(\|u\|_{\rho,1}+\|Cu\|_{\rho,1/2})+\|u\|_{\rho,1/2}\|Cu\|_{\rho,1/2}+\|Cu\|_{\rho,1/2}^{2}.\label{eq:u_1}
\end{equation}
Next, we estimate the norm $\|Cu\|_{\rho,1/2}.$ First, we compute
using the positive definiteness estimate for $\mathcal{N}_{11}^{-1}$
in $H_{\rho}^{1/2}(\R;H_{1})$ from \prettyref{lem:N11}
\begin{align*}
\|Cu\|_{1/2}^{2} & \lesssim\Re\langle\mathcal{N}_{11}^{-1}Cu,Cu\rangle_{\rho,1/2}\\
 & =\Re\langle\partial_{t,\rho}^{1/2}\mathcal{N}_{11}^{-1}Cu,\partial_{t,\rho}^{1/2}Cu\rangle_{\rho,0}\\
 & =\Re\langle\partial_{t,\rho}^{1/2}\mathcal{N}_{11}^{-1}g-\partial_{t,\rho}^{1/2}\mathcal{N}_{11}^{-1}\mathcal{N}_{10}u-\partial_{t,\rho}^{1/2}v,\partial_{t,\rho}^{1/2}Cu\rangle_{\rho,0}\\
 & =\Re\langle\partial_{t,\rho}^{1/2}\mathcal{N}_{11}^{-1}\partial_{t,\rho}^{-1/2}\partial_{t,\rho}^{1/2}g-\partial_{t,\rho}^{1/2}\mathcal{N}_{11}^{-1}\mathcal{N}_{10}\partial_{t,\rho}^{-1/2}\partial_{t,\rho}^{1/2}u,\partial_{t,\rho}^{1/2}Cu\rangle_{\rho,0}\\
 & \quad-\Re\langle\partial_{t,\rho}^{1/2}v,\partial_{t,\rho}^{1/2}Cu\rangle_{\rho,0}\\
 & \lesssim(\|g\|_{\rho,1/2}+\|u\|_{\rho,1/2})\|Cu\|_{\rho,1/2}-\Re\langle\partial_{t,\rho}^{1/2}v,\partial_{t,\rho}^{1/2}Cu\rangle_{\rho,0}.
\end{align*}
Moreover, we compute using \prettyref{lem:M00}
\begin{align*}
-\Re\langle\partial_{t,\rho}^{1/2}v,\partial_{t,\rho}^{1/2}Cu\rangle_{\rho,0} & =\Re\langle-\partial_{t,\rho}^{1/2}C^{\ast}v,\partial_{t,\rho}^{1/2}u\rangle_{\rho,0}\\
 & =\Re\langle f-\partial_{t,\rho}\mathcal{M}_{00}u-\mathcal{N}_{00}u-\mathcal{N}_{01}v,|\partial_{t,\rho}|u\rangle_{\rho,0}\\
 & \lesssim\left(\|f\|_{\rho,0}+\|u\|_{\rho,0}+\|v\|_{\rho,0}\right)\|u\|_{\rho,1}-\Re\langle\partial_{t,\rho}\mathcal{M}_{00}u,|\partial_{t,\rho}|u\rangle_{\rho,0}\\
 & \lesssim\left(\|f\|_{\rho,0}+\|u\|_{\rho,0}+\|v\|_{\rho,0}\right)\|u\|_{\rho,1}+\|u\|_{\rho,1/2}^{2}.
\end{align*}
Summarising, we obtain 
\begin{align*}
\|Cu\|_{\rho,1/2}^{2} & \lesssim(\|g\|_{\rho,1/2}+\|u\|_{\rho,1/2})\|Cu\|_{\rho,1/2}+\left(\|f\|_{\rho,0}+\|u\|_{\rho,0}+\|v\|_{\rho,0}\right)\|u\|_{\rho,1}+\|u\|_{\rho,1/2}^{2}\\
 & \lesssim(\|(f,g)\|_{\rho,(0,1/2)}+\|u\|_{\rho,1/2})\|Cu\|_{\rho,1/2}+\|(f,g)\|_{\rho,(0,1/2)}\|u\|_{\rho,1}+\|u\|_{\rho,1/2}^{2},
\end{align*}
which is a quadratic inequality in $\|Cu\|_{\rho,1/2}\geq0$ and yields
\begin{align}
\|Cu\|_{\rho,1/2} & \lesssim\|(f,g)\|_{\rho,(0,1/2)}+\|u\|_{\rho,1/2}+\sqrt{\|(f,g)\|_{\rho,(0,1/2)}\|u\|_{\rho,1}+\|u\|_{\rho,1/2}^{2}}\nonumber \\
 & \lesssim\|(f,g)\|_{\rho,(0,1/2)}+\|u\|_{\rho,1/2}+\sqrt{\|(f,g)\|_{0,1/2}}\sqrt{\|u\|_{\rho,1}}.\label{eq:Cu1/2}
\end{align}
Next, observe that 
\begin{align*}
\|u\|_{\rho,1/2}^{2} & =\langle\partial_{t,\rho}^{1/2}u,\partial_{t,\rho}^{1/2}u\rangle_{\rho,0}=\langle|\partial_{t,\rho}|u,u\rangle_{\rho,0}\leq\|u\|_{\rho,1}\|u\|_{\rho,0}\lesssim\|u\|_{\rho,1}\|(f,g)\|_{\rho,(0,1/2)}.
\end{align*}
Using this inequality, the one derived in \prettyref{eq:u_1}, and
\prettyref{eq:Cu1/2}, we obtain 
\begin{align*}
\|u\|_{\rho,1}^{2} & \lesssim\|(f,g)\|_{\rho,(0,1/2)}\|u\|_{\rho,1}\\
 & \quad+\|(f,g)\|_{\rho,(0,1/2)}\left(\|(f,g)\|_{\rho,(0,1/2)}+\|u\|_{\rho,1/2}+\sqrt{\|(f,g)\|_{\rho,(0,1/2)}}\sqrt{\|u\|_{\rho,1}}\right)+\\
 & \quad+\|u\|_{\rho,1/2}\left(\|(f,g)\|_{\rho,(0,1/2)}+\|u\|_{\rho,1/2}+\sqrt{\|(f,g)\|_{\rho,(0,1/2)}}\sqrt{\|u\|_{\rho,1}}\right)\\
 & \quad+\|(f,g)\|_{\rho,(0,1/2)}^{2}+\|u\|_{\rho,1/2}^{2}+\|(f,g)\|_{\rho,(0,1/2)}\|u\|_{\rho,1}\\
 & \lesssim\|(f,g)\|_{\rho,(0,1/2)}\|u\|_{\rho,1}+\|(f,g)\|_{\rho,(0,1/2)}^{3/2}\|u\|_{\rho,1}^{1/2}+\|(f,g)\|_{\rho,(0,1/2)}^{2}.
\end{align*}
Employing \prettyref{lem:estimate_root} (applied to $x_{0}=\|u\|_{\rho,1}^{1/2}$)
we get that 
\[
\|u\|_{\rho,1}\leq F(\|(f,g)\|_{\rho,(0,1/2)}),
\]
where $F:\R_{\geq0}\to\R_{\geq0}$ is continuous with $F(0)=0.$ This
proves that 
\begin{multline*}
\mathcal{S}_{\rho}:(\partial_{t,\rho}\mathcal{M}+\mathcal{N}+A)[H_{\rho}^{2}(\R;\dom(A))]\subseteq L_{2,\rho}(\R;H_{0})\times H_{\rho}^{1/2}(\R;H_{1})\\
\to H_{\rho}^{1}(\R;H_{0})\times L_{2,\rho}(\R;H_{1})
\end{multline*}
is continuous at $0$ and hence, bounded. Since $(\partial_{t,\rho}\mathcal{M}+\mathcal{N}+A)[H_{\rho}^{2}(\R;\dom(A))]$
is dense in $L_{2,\rho}(\R;H_{0})\times H_{\rho}^{1/2}(\R;H_{1})$
by \prettyref{cor:densityresult}, the first regularity statement
holds. The additional regularity $Cu\in H_{\rho}^{1/2}$ then follows
from estimate \prettyref{eq:Cu1/2}.
\end{proof}

\section{Applications\label{sec:Applications}}

\subsection{Maximal regularity and bounded commutators}

In this section, we will apply our main result \prettyref{thm:main}
to prove maximal regularity for a broad class of evolutionary equations.
Note that the second main theorem of the present manuscript is concerned
with the case, where well-posedness in $H_{\rho}^{1/2}$ is obtained
by a bounded commutator assumption involving $\mathcal{N}$ and by
restricting $\mathcal{M}$ to the case commuting with $\partial_{t,\rho}^{-1}$.
It turns out that this situation is closer to the applications as
we shall outline below.

As above, we assume that $H_{0}$ and $H_{1}$ are two complex Hilbert
spaces and we set $H\coloneqq H_{0}\oplus H_{1}.$ Moreover, $A=\left(\begin{array}{cc}
0 & -C^{*}\\
C & 0
\end{array}\right)$ for some densely defined closed linear operator $C:\dom(C)\subseteq H_{0}\to H_{1}$
and $\mathcal{M}$ and $\mathcal{N}$ have the form $\mathcal{M}=\left(\begin{array}{cc}
\mathcal{M}_{00} & 0\\
0 & 0
\end{array}\right)$ as well as $\mathcal{N}=\left(\begin{array}{cc}
\mathcal{N}_{00} & \mathcal{N}_{01}\\
\mathcal{N}_{10} & \mathcal{N}_{11}
\end{array}\right)$ with appropriate linear operators in $L(L_{2,\rho}(\R;H_{j}),L_{2,\rho}(\R;H_{i}))$,
$i,j\in\{0,1\}$. 
\begin{thm}
\label{thm:max_reg_commuator}Let $A,\mathcal{M}$ and $\mathcal{N}$
be as in \prettyref{eq:A is C star C}, \prettyref{eq:M is M00},
and \prettyref{eq:N is Matrix}. Assume, in addition, that $\partial_{t,\rho}^{-1}\mathcal{M}=\mathcal{M}\partial_{t,\rho}^{-1}$.
Moreover, we assume that there is $c>0$ such that
\[
\Re\langle\mathcal{M}_{00}\phi_{0},\phi_{0}\rangle_{\rho,0}\geq c\langle\phi_{0},\phi_{0}\rangle_{\rho,0}\quad(\phi\in L_{2,\rho}(\mathbb{R};H_{0}))
\]

and
\[
\Re\langle\left(\partial_{t,\rho}\mathcal{M}+\mathcal{N}\right)\phi,\phi\rangle_{\rho,0}\geq c\langle\phi,\phi\rangle_{\rho,0},
\]

for all $\phi\in H_{\rho}^{3/2}(\R;H).$ Finally, we assume that there
is $0\leq\tilde{c}<c$ and $d>0$ such that
\[
\|[\partial_{t,\rho}^{1/2},\mathcal{N}]\phi\|_{\rho,0}\leq\tilde{c}\|\phi\|_{\rho,1/2}+d\|\phi\|_{\rho,0}\quad(\phi\in H_{\rho}^{1/2}(\R;H)).
\]
Then 
\begin{align*}
 & \mathcal{S}_{\rho}[L_{2,\rho}(\R;H_{0})\times H_{\rho}^{1/2}(\R;H_{1})]\\
 & \subseteq\left(H_{\rho}^{1}(\R;H_{0})\cap H_{\rho}^{1/2}(\R;\dom(C))\right)\times\left(L_{2,\rho}(\R;\dom(C^{\ast}))\cap H_{\rho}^{1/2}(\R;H_{1})\right),
\end{align*}
where $\mathcal{S}_{\rho}\coloneqq\left(\overline{\partial_{t,\rho}\mathcal{M}+\mathcal{N}+A}\right)^{-1}\in L(L_{2,\rho}(\R;H))$.
\end{thm}

\begin{rem}
\label{rem:oncecommualways}The condition $\partial_{t,\rho}^{-1}\mathcal{M}=\mathcal{M}\partial_{t,\rho}^{-1}$
implies $\partial_{t,\rho}^{-1/2}\mathcal{M}=\mathcal{M}\partial_{t,\rho}^{-1/2}$.
Indeed, to start off with, the fact that $\mathcal{M}$ commutes with
$\partial_{t,\rho}^{-1}$ yields 
\begin{align*}
\mathcal{M}\partial_{t,\rho} & \subseteq\partial_{t,\rho}\mathcal{M}.
\end{align*}
Hence, $\mathcal{M}\left(\partial_{t,\rho}-2\rho\right)\subseteq\left(\partial_{t,\rho}-2\rho\right)\mathcal{M}$
and, thus, using \prettyref{prop:elementary facts td} (c), we infer
\[
-\left(\partial_{t,\rho}^{*}\right)^{-1}\mathcal{M}=\left(\partial_{t,\rho}-2\rho\right)^{-1}\mathcal{M}\subseteq\mathcal{M}\left(\partial_{t,\rho}-2\rho\right)^{-1}=-\mathcal{M}\left(\partial_{t,\rho}^{*}\right)^{-1}.
\]
Since $\left(\partial_{t,\rho}^{*}\right)^{-1}\mathcal{M}$ is defined
everywhere, we deduce $\left(\partial_{t,\rho}^{*}\right)^{-1}\mathcal{M}=\mathcal{M}\left(\partial_{t,\rho}^{*}\right)^{-1}$.
Next, by the approximation theorem of Weierstraß, we have that the
polynomials in $z$ and $z^{*}$ as continuous functions on $C(V)$
are dense in $C(V)$ endowed with the $\sup$-norm, where $V\coloneqq\overline{B_{\mathbb{C}}(1/(2\rho),1/(2\rho))}$.
Hence, we find a sequence of polynomials $\left(z\mapsto p_{n}(z,z^{*})\right)_{n}$
in $z$ and $z^{*}$ such that $p_{n}\to(z\mapsto\sqrt{z})$ uniformly
on $V$ as $n\to\infty$. In consequence, using the Fourier--Laplace
transformation, we obtain that $p_{n}\left(\partial_{t,\rho}^{-1},\left(\partial_{t,\rho}^{*}\right)^{-1}\right)\to\partial_{t,\rho}^{-1/2}$
in $L(L_{2,\rho}(\mathbb{R};H)).$ Thus, we infer using the commutator
properties of $\mathcal{M}$ shown above 
\[
\partial_{t,\rho}^{-1/2}\mathcal{M}=\lim_{n\to\infty}p_{n}\left(\partial_{t,\rho}^{-1},\left(\partial_{t,\rho}^{*}\right)^{-1}\right)\mathcal{M}=\lim_{n\to\infty}\mathcal{M}p_{n}\left(\partial_{t,\rho}^{-1},\left(\partial_{t,\rho}^{*}\right)^{-1}\right)=\mathcal{M}\partial_{t,\rho}^{-1/2}.
\]
\end{rem}

\begin{proof}[Proof of \prettyref{thm:max_reg_commuator}]
 Let $f\in L_{2,\rho}(\R;H_{0})$ and $g\in H_{\rho}^{1/2}(\R;H_{1}).$
Moreover, let $(u,v)=\mathcal{S}_{\rho}(f,g)\in L_{2,\rho}(\R;H).$
We choose $0<\varepsilon<c-\tilde{c}$ and $\delta\geq\frac{d^{2}}{4\varepsilon\sqrt{\rho}}$
and consider the operator 
\[
\tilde{\mathcal{N}}\coloneqq\mathcal{N}+\delta\partial_{t,\rho}^{-1/2}.
\]
It is clear that 
\[
\overline{\left(\partial_{t,\rho}\mathcal{M}+\tilde{\mathcal{N}}+A\right)}\left(\begin{array}{c}
u\\
v
\end{array}\right)=\left(\begin{array}{c}
f\\
g
\end{array}\right)+\delta\partial_{t,\rho}^{-1/2}\left(\begin{array}{c}
u\\
v
\end{array}\right)\in L_{2,\rho}(\R;H_{0})\times H_{\rho}^{1/2}(\R;H)
\]
and hence, to show the claim, it suffices to prove that the operators
$\mathcal{M}$ and $\mathcal{\tilde{N}}$ satisfy the assumptions
of \prettyref{thm:main}. We first note that $\mathcal{M}'=0$ and
that \prettyref{eq:M00pd} holds by assumption and \prettyref{eq:pd0}
follows from the inequality assumed for $\partial_{t,\rho}\mathcal{M}+\mathcal{N}$
and the fact that $\Re\partial_{t,\rho}^{-1/2}\geq0.$ Thus, it remains
to show \prettyref{eq:pd12}. For doing so, let $\phi\in H_{\rho}^{3/2}(\R;H).$
Since $\mathcal{M}$ commutes with $\partial_{t,\rho}^{-1}$, it follows
that it also commutes with $\partial_{t,\rho}^{-1/2}$ (see \prettyref{rem:oncecommualways})
and thus
\[
\partial_{t,\rho}^{3/2}\mathcal{M}\phi=\partial_{t,\rho}^{3/2}\mathcal{M}\partial_{t,\rho}^{-1/2}\partial_{t,\rho}^{1/2}\phi=\partial_{t,\rho}\mathcal{M}\partial_{t,\rho}^{1/2}\phi.
\]

Hence, we can compute 
\begin{align*}
 & \Re\langle\left(\partial_{t,\rho}\mathcal{M}+\tilde{\mathcal{N}}\right)\phi,\phi\rangle_{\rho,1/2}\\
 & =\Re\langle\partial_{t,\rho}^{3/2}\mathcal{M}\phi+\partial_{t,\rho}^{1/2}\tilde{\mathcal{N}}\phi,\partial_{t,\rho}^{1/2}\phi\rangle_{\rho,0}\\
 & =\Re\langle\partial_{t,\rho}\mathcal{M}\partial_{t,\rho}^{1/2}\phi+\partial_{t,\rho}^{1/2}\mathcal{N}\phi,\partial_{t,\rho}^{1/2}\phi\rangle_{\rho,0}+\delta\Re\langle\phi,\partial_{t,\rho}^{1/2}\phi\rangle_{\rho,0}\\
 & \geq\Re\langle(\partial_{t,\rho}\mathcal{M}+\mathcal{N})\partial_{t,\rho}^{1/2}\phi,\partial_{t,\rho}^{1/2}\phi\rangle_{\rho,0}+\Re\langle[\partial_{t,\rho}^{1/2},\mathcal{N}]\phi,\partial_{t,\rho}^{1/2}\phi\rangle_{\rho,0}+\sqrt{\rho}\delta\|\phi\|_{\rho,0}^{2}\\
 & \geq c\|\phi\|_{\rho,1/2}^{2}-\|[\partial_{t,\rho},\mathcal{N}]\phi\|_{\rho,0}\|\phi\|_{\rho,1/2}+\sqrt{\rho}\delta\|\phi\|_{\rho,0}^{2}\\
 & \geq\left(c-\tilde{c}\right)\|\phi\|_{\rho,1/2}^{2}-d\|\phi\|_{\rho,0}\|\phi\|_{\rho,1/2}+\sqrt{\rho}\delta\|\phi\|_{\rho,0}^{2}\\
 & \geq(c-\tilde{c}-\varepsilon)\|\phi\|_{\rho,1/2}^{2}+(\sqrt{\rho}\delta-\frac{d^{2}}{4\varepsilon})\|\phi\|_{\rho,0}^{2}\\
 & \geq(c-\tilde{c}-\varepsilon)\|\phi\|_{\rho,1/2}^{2},
\end{align*}
which shows \prettyref{eq:pd12}.
\end{proof}
If the coefficient operators, $\mathcal{M}$ and $\mathcal{N}$, act
in a `physically meaningful manner'; that is, if they are causal (see
definition below), then the latter result (as well as the other main
result \prettyref{thm:main}) also imply a maximal regularity result
locally in time. For this we need a closer look into the well-posedness
result \prettyref{thm:sol_theory_basic}, which in turn prerequisites
the following notion. We define
\[
S_{c}(\mathbb{R};H)\coloneqq\lin\{f\colon\mathbb{R}\to H;f\text{ simple function with compact support}\}.
\]

\begin{defn}
Let $K_{0},K_{1}$ be Hilbert spaces, $\rho_{0}\in\mathbb{R}$, and
\[
\mathcal{C}\colon S_{c}(\mathbb{R};K_{0})\to\bigcap_{\rho\geq\rho_{0}}L_{2,\rho}(\mathbb{R};K_{1})
\]
linear. Then we call $\mathcal{C}$ \emph{evolutionary (at $\rho_{0}$),}
if, for all $\rho\geq\rho_{0}$, $\mathcal{C}$ admits a continuous
extension $\mathcal{C}_{\rho}\in L(L_{2,\rho}(\mathbb{R};K_{0}),L_{2,\rho}(\mathbb{R};K_{1}))$
satisfying
\[
\sup_{\rho\geq\rho_{0}}\|\mathcal{C}_{\rho}\|_{L(L_{2,\rho}(\mathbb{R};K_{0}),L_{2,\rho}(\mathbb{R};K_{1}))}<\infty.
\]
 We gather two results important for evolutionary mappings of the
type discussed in the latter definition. For the intricacies of the
interplay of causality and closure of operators, we refer to \cite{W15_CB}.
\end{defn}

\begin{prop}[{\cite[Remark 2.1.5]{W16_H} or \cite[Lemma 4.2.5 (a)]{seifert2020evolutionary}}]
Let $K_{0},K_{1}$ be Hilbert spaces, $\rho_{0}\in\mathbb{R}$, and
$\mathcal{C}\colon S_{c}(\mathbb{R};K_{0})\to\bigcap_{\rho\geq\rho_{0}}L_{2,\rho}(\mathbb{R};K_{1})$
linear and $\mathcal{C}$ evolutionary at $\rho_{0}$. Then $\mathcal{C}_{\rho}$
is \emph{causal} for all $\rho\geq\rho_{0}$; that is, for all $t\in\mathbb{R}$
and $f\in L_{2,\rho}(\mathbb{R};K_{0})$ we have
\[
\spt f\subseteq[t,\infty)\Rightarrow\spt\mathcal{C}_{\rho}f\subseteq[t,\infty).
\]
 
\end{prop}

\begin{thm}[{\cite[Theorem 3.4.6]{W16_H} or \cite[Theorem 3.4]{Waurick2015}}]
In addition to the assumptions in \prettyref{thm:sol_theory_basic},
assume that $\mathcal{M},\mathcal{M}'$ and $\mathcal{N}$ are evolutionary.
Then 
\[
\mathcal{S}_{\rho}=\left(\overline{\partial_{t,\rho}\mathcal{M}_{\rho}+\mathcal{N}_{\rho}+A}\right)^{-1}\in L(L_{2,\rho}(\mathbb{R};H))
\]
 is causal.
\end{thm}

Having presented the remaining technical ingredients for the localisation
on bounded time-intervals, we can present the local maximal regularity
statement next.
\begin{cor}
\label{cor:local-in-time}In addition to the assumptions in \prettyref{thm:max_reg_commuator},
assume that $\mathcal{M}$ and $\mathcal{N}$ are evolutionary. Let
$T\in]0,\infty[$. Then there exists $\kappa\geq0$ such that for
all $f\in L_{2,\rho}(\mathbb{R};H_{0})$ and $g\in H_{\rho}^{1/2}(\mathbb{R};H_{1})$
with $\spt f,\spt g\subseteq[0,T[$, $(u,v)=\mathcal{S}_{\rho}(f,g)$
are supported in $[0,\infty)$ only and satisfy
\[
\|u\|_{H^{1}[0,T)}+\|Cu\|_{H^{1/2}[0,T)}+\|C^{*}v\|_{L_{2}[0,T)}+\|v\|_{H^{1/2}[0,T)}\leq\kappa\left(\|f\|_{L_{2}[0,T)}+\|g\|_{H^{1/2}[0,T)}\right).
\]
\end{cor}

\begin{proof}
Firstly, observe that $H_{\rho}^{1/2}(\mathbb{R};H)\ni u\mapsto u|_{[0,T)}\in H^{1/2}([0,T);H)$
continuously, by complex interpolation, see also \prettyref{rem:interpolation}.
Next, let $\phi\in C_{c}^{\infty}(\mathbb{R})$ with $0\leq\phi\leq1$
have the following properties
\[
\phi(t)=\begin{cases}
0, & t\leq-T/2,\\
1 & 0\leq t\leq T,\\
0 & t\geq3T/2.
\end{cases}
\]
For $h\in L_{2}(]0,T[;H)$ define
\[
\hat{h}\coloneqq\phi(\cdot)\begin{cases}
h(-\cdot), & \text{ on }]-\infty,0],\\
h(\cdot), & \text{ on }]0,T[,\\
h(T-\cdot) & \text{ on }[T,\infty[.
\end{cases}
\]
Then it is not difficult to see that $E\colon h\mapsto\hat{h}$ is
continuous as a mapping from $L_{2}(]0,T[;H)$ to $L_{2}(\mathbb{R};H)$
and as a mapping from $H^{1}(]0,T[;H)$ to $H_{\rho}^{1}(\mathbb{R};H)$.
Thus, by interpolation, we infer continuity as a mapping from $H^{1/2}(]0,T[;H)$
to $H_{\rho}^{1/2}(\mathbb{R};H).$ Thus, there exists $\kappa\geq0$
such that for all $g\in H_{\rho}^{1/2}(\mathbb{R};H)$ with $\spt g\subseteq[0,T[$
we have $\|g\|_{\rho,\frac{1}{2}}\leq\kappa\|g\|_{H^{1/2}[0,T[}.$
This estimate together with \prettyref{thm:max_reg_commuator} then
implies the assertion.
\end{proof}
\begin{rem}
Note that a prototype of evolutionary operators are operators defined
as multiplication by a function, see also \cite[Example 2.1.1]{W16_H}.
This prototype will be discussed next.
\end{rem}

\subsection{Commutators for multiplication operators}

In this subsection we inspect the conditions on the operator $\mathcal{N}$
assumed in \prettyref{thm:max_reg_commuator} for the concrete case
of $\mathcal{N}$ being a multiplication operator. More precisely,
we assume the following: Let $N:\R\to L(H)$ be a strongly measurable
bounded mapping. Then $N$ induces an evolutionary operator 
\[
\mathcal{N}:S_{c}(\mathbb{R};H)\to\bigcap_{\rho\geq0}L_{2,\rho}(\R;H),\quad f\mapsto\left(t\mapsto N(t)f(t)\right)
\]
with $\|\mathcal{N}_{\rho}\|_{L(L_{2,\rho}(\R;H))}=\|N\|_{\infty}$
for all $\rho\geq0.$ Note that all continuous extensions $\mathcal{N}_{\rho}$,
$\rho\geq0$, act as multiplication by $N$. We start to give a representation
for the term $\partial_{t,\rho}^{1/2}\phi$ for regular functions
$\phi.$
\begin{lem}
\label{lem:fractional_derivative}Let $\rho>0$ and $\phi\in H_{\rho}^{1}(\R;H)$.
Then 
\[
\left(\partial_{t,\rho}^{1/2}\phi\right)(t)=\frac{1}{2\Gamma(1/2)}\int_{-\infty}^{t}(t-s)^{-3/2}(\phi(t)-\phi(s))\d s\quad(t\in\R\text{ a.e.}).
\]
\end{lem}

\begin{proof}
We have $\partial_{t,\rho}^{1/2}\phi=\partial_{t,\rho}^{-1/2}\partial_{t,\rho}\phi=\partial_{t,\rho}^{-1/2}\phi'$
(see also \prettyref{prop:fractional_integral}) and thus, by \prettyref{prop:fractional_integral}
\begin{align*}
\left(\partial_{t,\rho}^{1/2}\phi\right)(t) & =\frac{1}{\Gamma(1/2)}\int_{-\infty}^{t}(t-s)^{-1/2}\phi'(s)\d s\\
 & =\frac{1}{\Gamma(1/2)}\int_{0}^{\infty}s{}^{-1/2}\phi'(t-s)\d s\\
 & =\frac{1}{2\Gamma(1/2)}\int_{0}^{\infty}\int_{s}^{\infty}r^{-3/2}\d r\:\phi'(t-s)\d s\\
 & =\frac{1}{2\Gamma(1/2)}\int_{0}^{\infty}r^{-3/2}\int_{0}^{r}\phi'(t-s)\d s\d r\\
 & =\frac{1}{2\Gamma(1/2)}\int_{0}^{\infty}r^{-3/2}(\phi(t)-\phi(t-r))\d r\\
 & =\frac{1}{2\Gamma(1/2)}\int_{-\infty}^{t}(t-s)^{-3/2}(\phi(t)-\phi(s))\d s\quad(t\in\R\text{ a.e.}).\tag*{\qedhere}
\end{align*}
\end{proof}
Using this expression, we can prove our first result on commutators
with the fractional derivative.
\begin{prop}
\label{prop:bdd_commut}Let $\rho_{0}>0$ and assume that $[\mathcal{N}_{\rho_{0}},\partial_{t,\rho_{0}}^{1/2}]$
is bounded as an operator on $L_{2,\rho_{0}}(\R;H).$ Then for each
$\rho\geq\rho_{0}$ we have $\mathcal{N}_{\rho}[H_{\rho}^{1/2}(\R;H)]\subseteq H_{\rho}^{1/2}(\R;H)$
and 
\[
\|[\mathcal{N}_{\rho},\partial_{t,\rho}^{1/2}]\|_{L(L_{2,\rho}(\R;H))}\leq\|[\mathcal{N}_{\rho_{0}},\partial_{t,\rho_{0}}^{1/2}]\|_{L(L_{2,\rho_{0}}(\R;H))}+2\|N\|_{\infty}\left(\sqrt{\rho}-\sqrt{\rho_{0}}\right).
\]
\end{prop}

\begin{proof}
Let $\rho\geq\rho_{0}$. To start off with, we prove the following
statement:
\[
\phi\in H_{\rho}^{1/2}(\R;H)\,\Leftrightarrow\,\e^{(\rho_{0}-\rho)\cdot}\phi\in H_{\rho_{0}}^{1/2}(\R;H).
\]
For this, let $\phi\in H_{\rho}^{1/2}(\R;H)$; that is, $(\i\m+\rho)^{\frac{1}{2}}\mathcal{L}_{\rho}\phi\in L_{2}(\R;H)$.
Note that $\mathcal{L}_{\rho_{0}}\e^{(\rho_{0}-\rho)\cdot}\phi=\mathcal{L}_{\rho}\phi$
and hence, it suffices to show that 
\[
(\i\m+\rho_{0})^{\frac{1}{2}}\mathcal{L}_{\rho}\phi\in L_{2}(\R;H).
\]
The latter however is clear, since $\left(t\mapsto(\i t+\rho_{0})^{1/2}(\i t+\rho)^{-1/2}\right)\in L_{\infty}(\R;H).$
Since this argument is completely symmetric in $\rho$ and $\rho_{0}$,
the asserted equivalence holds. \\
If now $\phi\in H_{\rho}^{1/2}(\R;H)$, we infer that 
\[
\e^{(\rho_{0}-\rho)\cdot}\mathcal{N}_{\rho}\phi=\mathcal{N}_{\rho_{0}}\e^{(\rho_{0}-\rho)\cdot}\phi\in H_{\rho_{0}}^{1/2}(\R;H)
\]
and thus, $\mathcal{N}_{\rho}[H_{\rho}^{1/2}(\R;H)]\subseteq H_{\rho}^{1/2}(\R;H)$.
To compute the norm of the commutator, let $\phi\in C_{c}^{\infty}(\R;H).$
Then we estimate 
\begin{align*}
\|[\mathcal{N}_{\rho},\partial_{t,\rho}^{1/2}]\phi\|_{\rho,0} & =\|\e^{(\rho_{0}-\rho)\cdot}\left(\mathcal{N}_{\rho}\partial_{t,\rho}^{1/2}-\partial_{t,\rho}^{1/2}\mathcal{N}_{\rho}\right)\phi\|_{\rho_{0},0}\\
 & =\|\mathcal{N}_{\rho_{0}}\e^{(\rho_{0}-\rho)\cdot}\partial_{t,\rho}^{1/2}\phi-\e^{(\rho_{0}-\rho)\cdot}\partial_{t,\rho}^{1/2}\mathcal{N}_{\rho}\phi\|_{\rho_{0},0}\\
 & \leq\|[\mathcal{N}_{\rho_{0}},\partial_{t,\rho_{0}}^{1/2}]\e^{(\rho_{0}-\rho)\cdot}\phi\|_{\rho_{0},0}+\|\mathcal{N}_{\rho_{0}}\left(\e^{(\rho_{0}-\rho)\cdot}\partial_{t,\rho}^{1/2}-\partial_{t,\rho_{0}}^{1/2}\e^{(\rho_{0}-\rho)\cdot}\right)\phi\|_{\rho_{0},0}+\\
 & \quad+\|\left(\e^{(\rho_{0}-\rho)\cdot}\partial_{t,\rho}^{1/2}-\partial_{t,\rho_{0}}^{1/2}\e^{(\rho_{0}-\rho)\cdot}\right)\mathcal{N}_{\rho}\phi\|_{\rho_{0},0}.
\end{align*}
Hence, we need to estimate the norm of the operator $\e^{(\rho_{0}-\rho)\cdot}\partial_{t,\rho}^{1/2}-\partial_{t,\rho_{0}}^{1/2}\e^{(\rho_{0}-\rho)\cdot}$.
For this, let $\psi\in C_{c}^{\infty}(\R;H)$ and compute, using \prettyref{lem:fractional_derivative}
\begin{align*}
 & (\e^{(\rho_{0}-\rho)\cdot}\partial_{t,\rho}^{1/2}\psi-\partial_{t,\rho_{0}}^{1/2}\e^{(\rho_{0}-\rho)\cdot}\psi)(t)\\
 & =\e^{(\rho_{0}-\rho)t}\frac{1}{2\Gamma(1/2)}\int_{-\infty}^{t}(t-s)^{-3/2}(\psi(t)-\psi(s))\d s-\\
\quad & -\frac{1}{2\Gamma(1/2)}\int_{-\infty}^{t}(t-s)^{-3/2}(\e^{(\rho_{0}-\rho)t}\psi(t)-\e^{(\rho_{0}-\rho)s}\psi(s))\d s\\
 & =\frac{1}{2\Gamma(1/2)}\int_{-\infty}^{t}(t-s)^{-3/2}\left(\e^{(\rho_{0}-\rho)s}-\e^{(\rho_{0}-\rho)t}\right)\psi(s)\d s\\
 & =\frac{1}{2\Gamma(1/2)}\int_{-\infty}^{t}(t-s)^{-3/2}\left(1-\e^{(\rho_{0}-\rho)(t-s)}\right)\e^{(\rho_{0}-\rho)s}\psi(s)\d s\\
 & =(k_{\rho_{0}-\rho}\ast\e^{(\rho_{0}-\rho)\cdot}\psi)(t),
\end{align*}
where $k_{\mu}(t)\coloneqq\frac{1}{2\Gamma(1/2)}\chi_{\R_{\geq0}}(t)t^{-3/2}(1-\e^{\mu t})$
for $t,\mu\in\R.$ Note that for $\mu=(\rho_{0}-\rho)\leq0$ the function
$k_{\rho_{0}-\rho}$ is positive and hence, using the convolution
theorem, we obtain
\[
\|k_{\rho_{0}-\rho}\ast\|_{L(L_{2,\rho_{0}}(\R;H))}=\int_{\R}k_{\rho_{0}-\rho}(t)\e^{-\rho_{0}t}\d t.
\]
We compute using the integral representation of the $\Gamma$-function
\begin{align*}
\int_{\R}k_{\rho_{0}-\rho}(t)\e^{-\rho_{0}t}\d t & =\frac{1}{2\Gamma(1/2)}\int_{0}^{\infty}t^{-3/2}(1-\e^{(\rho_{0}-\rho)t})\e^{-\rho_{0}t}\d t\\
 & =\frac{1}{\Gamma(1/2)}\int_{0}^{\infty}t^{-1/2}\left(-\rho_{0}\e^{-\rho_{0}t}+\rho\e^{-\rho t}\right)\d t\\
 & =\sqrt{\rho}-\sqrt{\rho_{0}}
\end{align*}
and thus, 
\begin{equation}
\|(\e^{(\rho_{0}-\rho)\cdot}\partial_{t,\rho}^{1/2}-\partial_{t,\rho_{0}}^{1/2}\e^{(\rho_{0}-\rho)\cdot})\phi\|_{\rho_{0},0}\leq\left(\sqrt{\rho}-\sqrt{\rho_{0}}\right)\|\e^{(\rho_{0}-\rho)\cdot}\phi\|_{\rho_{0},0}=\left(\sqrt{\rho}-\sqrt{\rho_{0}}\right)\|\phi\|_{\rho,0}.\label{eq:commutator_exponential}
\end{equation}
Summarising, we obtain the estimate 
\[
\|[\mathcal{N}_{\rho},\partial_{t,\rho}^{1/2}]\phi\|_{\rho,0}\leq\left(\|[\mathcal{N}_{\rho_{0}},\partial_{t,\rho_{0}}^{1/2}]\|_{L(L_{2,\rho_{0}}(\R;H))}+2\|N\|_{\infty}\left(\sqrt{\rho}-\sqrt{\rho_{0}}\right)\right)\|\phi\|_{\rho,0},
\]
which shows the claim. 
\end{proof}
The next proposition is devoted to the limit case $\rho_{0}=0$, which
is the case usually treated in the literature.
\begin{prop}
\label{prop:commutator_0}Assume that $[\mathcal{N}_{0},\partial_{t,0}^{1/2}]$
is bounded as an operator on $L_{2}(\R;H).$ Then for each $\rho\geq0$
we have $\mathcal{N}[H_{\rho}^{1/2}(\R;H)]\subseteq H_{\rho}^{1/2}(\R;H)$
and 
\[
\|[\mathcal{N_{\rho}},\partial_{t,\rho}^{1/2}]\|_{L(L_{2,\rho}(\R;H))}\leq\|[\mathcal{N}_{0},\partial_{t,0}^{1/2}]\|_{L(L_{2}(\R;H))}+2\|N\|_{\infty}\sqrt{\rho}.
\]
\end{prop}

\begin{proof}
Similar to the proof of \prettyref{prop:bdd_commut}, at first we
show 
\[
\phi\in H_{\rho}^{1/2}(\R;H)\iff\e^{-\rho\cdot}\phi\in H^{1/2}(\R;H).
\]
 For this, let $\phi\in H_{\rho}^{1/2}(\R;H).$ Then $(\i\m+\rho)^{1/2}\mathcal{L}_{\rho}\phi=(\i\m+\rho)^{1/2}\mathcal{F}\e^{-\rho\cdot}\phi\in L_{2}(\R;H).$
The latter implies $(\i\m)^{1/2}\mathcal{F}\e^{-\rho\cdot}\phi\in L_{2}(\R;H)$
and hence, $\e^{-\rho\cdot}\phi\in H^{1/2}(\R;H).$ If, on the other
hand, $\e^{-\rho\cdot}\phi\in H^{1/2}(\R;H)$, then $\phi\in L_{2,\rho}(\R;H)$
and $(\i\m)^{1/2}\mathcal{L}_{\rho}\phi=(\i\m)^{1/2}\mathcal{F}\e^{-\rho\cdot}\phi\in L_{2}(\R;H)$
and hence, 
\begin{align*}
 & \int_{\R}\|(\i t+\rho)^{1/2}\left(\mathcal{L}_{\rho}\phi\right)(t)\|_{H}^{2}\d t\\
 & =\int_{[-1,1]}\|(\i t+\rho)^{1/2}\left(\mathcal{L}_{\rho}\phi\right)(t)\|_{H}^{2}\d t+\int_{|t|>1}\|(\i t+\rho)^{1/2}\left(\mathcal{L}_{\rho}\phi\right)(t)\|_{H}^{2}\d t\\
 & \leq(1+\rho^{2})^{1/2}\|\phi\|_{L_{2,\rho}(\R;H)}^{2}+(1+\rho^{2})^{1/2}\int_{\R}\|(\i t)^{1/2}\left(\mathcal{L}_{\rho}\phi\right)(t)\|_{H}^{2}\d t<\infty
\end{align*}
and thus, $\phi\in H_{\rho}^{1/2}(\R;H).$ 

By the same argumentation as in the proof of \prettyref{prop:bdd_commut}
we infer $\mathcal{N}[H_{\rho}^{1/2}(\R;H)]\subseteq H_{\rho}^{1/2}(\R;H)$.
Following the lines of the proof of \prettyref{prop:bdd_commut},
we need to find an estimate for $\|\left(\e^{-\rho\cdot}\partial_{t,\rho}^{1/2}-\partial_{t,0}^{1/2}\e^{-\rho\cdot}\right)\phi\|_{0,0},$
for $\phi\in C_{c}^{\infty}(\R;H).$ The main problem in proving such
an estimate is that we do not have an explicit integral representation
for $\partial_{t,0}^{1/2}$ thus far. However, we have 
\begin{align*}
\partial_{t,0}^{1/2}\psi & =\mathcal{F}^{\ast}(\i\m)^{1/2}\mathcal{F}\psi\\
 & =\lim_{\rho_{0}\to0}\mathcal{F}^{\ast}(\i\m+\rho_{0})^{1/2}\mathcal{F}\psi\\
 & =\lim_{\rho_{0}\to0}\e^{-\rho_{0}\cdot}\mathcal{L}_{\rho_{0}}^{\ast}(\i\m+\rho_{0})^{1/2}\mathcal{L}_{\rho_{0}}\e^{\rho_{0}\cdot}\psi\\
 & =\lim_{\rho_{0}\to0}\e^{-\rho_{0}\cdot}\partial_{t,\rho_{0}}^{1/2}\e^{\rho_{0}\cdot}\psi
\end{align*}
with convergence in $L_{2}(\R;H)$ for each $\psi\in H^{1/2}(\R;H)$,
where we have used dominated convergence in the second line. Thus,
for $\phi\in C_{c}^{\infty}(\R;H)$ we have that 
\begin{align*}
\|\left(\e^{-\rho\cdot}\partial_{t,\rho}^{1/2}-\partial_{t,0}^{1/2}\e^{-\rho\cdot}\right)\phi\|_{0,0} & =\lim_{\rho_{0}\to0}\|\left(\e^{-\rho\cdot}\partial_{t,\rho}^{1/2}-\e^{-\rho_{0}\cdot}\partial_{t,\rho_{0}}^{1/2}\e^{(\rho_{0}-\rho)\cdot}\right)\phi\|_{0,0}\\
 & =\lim_{\rho_{0}\to0}\|\left(\e^{(\rho_{0}-\rho)\cdot}\partial_{t,\rho}^{1/2}-\partial_{t,\rho_{0}}^{1/2}\e^{(\rho_{0}-\rho)\cdot}\right)\phi\|_{\rho_{0},0}\\
 & \leq\lim_{\rho_{0}\to0}\left(\sqrt{\rho}-\sqrt{\rho_{0}}\right)\|\phi\|_{\rho,0}\\
 & =\sqrt{\rho}\|\phi\|_{\rho,0},
\end{align*}
where we have used \prettyref{eq:commutator_exponential}. Following
the lines of the proof of \prettyref{prop:bdd_commut} the assertion
follows.
\end{proof}
\begin{rem}
Note that \prettyref{thm:max_reg_commuator} in combination with \prettyref{prop:bdd_commut}
or \prettyref{prop:commutator_0} yields maximal regularity of the
corresponding evolutionary equation, if $\mathcal{N}$ has a bounded
commutator for some $\rho\geq0$. In particular, this covers the case
treated in \cite{Auscher_Egert2016} (see also \prettyref{subsec:Divergence-form-equations}
below).
\end{rem}

Our next goal is to prove the following proposition.
\begin{prop}
\label{prop:Zacher}Assume that 
\begin{equation}
C\coloneqq\int_{\R}\int_{\R}\frac{\|N(t)-N(s)\|^{2}}{|t-s|^{2+\delta}}\d t\d s<\infty\label{eq:Zacher}
\end{equation}
for some $\delta>0.$ Then 
\[
\forall\varepsilon>0\,\exists c>0\:\forall\phi\in H_{\rho}^{1/2}(\R;H):\;\|[\partial_{t,\rho}^{1/2},\mathcal{N}_{\rho}]\phi\|_{\rho,0}\leq\varepsilon\|\phi\|_{\rho,1/2}+c\|\phi\|_{\rho,0}.
\]
\end{prop}

\begin{rem}
Assumption \prettyref{eq:Zacher} is the main assumption imposed in
\cite[Corollary 1.1]{Dier_Zacher2017}.
\end{rem}

In order to prove \prettyref{prop:Zacher}, we want to apply \prettyref{lem:fractional_derivative}
to derive an integral expression for the commutator. Since \prettyref{lem:fractional_derivative}
just holds for functions in $H_{\rho}^{1}(\R;H),$ we need to regularise
$\mathcal{N}.$ 
\begin{lem}
\label{lem:regularise}For $\varepsilon>0$ we define 
\[
N_{\varepsilon}(t)\coloneqq\frac{1}{\varepsilon}\int_{t}^{t+\varepsilon}N(s)\d s\quad(t\in\R),
\]
where the integral is defined in the strong sense. We denote the associated
multiplication operator by $\mathcal{N}_{\varepsilon}$. Then the
following statements hold:

\begin{enumerate}[(a)]

\item For each $\varepsilon>0$ we have $\|N_{\varepsilon}\|_{\infty}\leq\|N\|_{\infty}$
and $\mathcal{N}_{\varepsilon,\rho}[H_{\rho}^{1}(\R;H)]\subseteq H_{\rho}^{1}(\R;H)$

\item $\mathcal{N}_{\varepsilon,\rho}\to\mathcal{N}_{\rho}$ strongly
in $L_{2,\rho}(\R;H)$ as $\varepsilon\to0$, 

\item If there exist $c_{1},c_{2}\geq0$ such that 
\[
\|[\partial_{t,\rho}^{1/2},\mathcal{N}_{\varepsilon,\rho}]\phi\|_{\rho,0}\leq c_{1}\|\phi\|_{\rho,1/2}+c_{2}\|\phi\|_{\rho,0}
\]
for each $\phi\in H_{\rho}^{1/2}(\R;H)$ and $\varepsilon>0$, then
$\mathcal{N}_{\rho}[H_{\rho}^{1/2}(\R;H)]\subseteq H_{\rho}^{1/2}(\R;H)$
with 
\[
\|[\partial_{t,\rho}^{1/2},\mathcal{N}_{\rho}]\phi\|_{\rho,0}\leq c_{1}\|\phi\|_{\rho,1/2}+c_{2}\|\phi\|_{\rho,0}
\]
for all $\phi\in H_{\rho}^{1/2}(\R;H).$

\end{enumerate}
\end{lem}

\begin{proof}
\begin{enumerate}[(a)]

\item Let $\varepsilon>0$. The estimate $\|N_{\varepsilon}\|_{\infty}\leq\|N\|_{\infty}$
is obvious. Let $u\in H_{\rho}^{1}(\R;H).$ In order to show $\mathcal{N}_{\varepsilon,\rho}u\in H_{\rho}^{1}(\R;H)$,
let $\phi\in C_{c}^{\infty}(\R).$ Then we compute 
\begin{align*}
 & \int_{\R}\left(\mathcal{N}_{\varepsilon,\rho}u\right)(t)\phi'(t)\d t\\
 & =\frac{1}{\varepsilon}\int_{\R}\int_{t}^{t+\varepsilon}N(s)u(t)\phi'(t)\d s\d t\\
 & =\frac{1}{\varepsilon}\int_{\R}\int_{0}^{\varepsilon}N(s+t)u(t)\phi'(t)\d s\d t\\
 & =\frac{1}{\varepsilon}\int_{0}^{\varepsilon}\int_{\R}N(s+t)u(t)\phi'(t)\d t\d s\\
 & =\frac{1}{\varepsilon}\int_{0}^{\varepsilon}\int_{\R}N(r)u(r-s)\phi'(r-s)\d r\d s\\
 & =\frac{1}{\varepsilon}\int_{\R}N(r)\int_{0}^{\varepsilon}u(r-s)\phi'(r-s)\d s\d r\\
 & =-\frac{1}{\varepsilon}\int_{\R}N(r)\int_{0}^{\varepsilon}\partial_{t,\rho}u(r-s)\phi(r-s)\d s\d r\\
 & \quad-\frac{1}{\varepsilon}\int_{\R}N(r)\left(u(r-\varepsilon)\phi(r-\varepsilon)-u(r)\phi(r)\right)\d r\\
 & =-\int_{\R}\left(\mathcal{N}_{\varepsilon,\rho}\partial_{t,\rho}u\right)(t)\phi(t)\d t-\frac{1}{\varepsilon}\int_{\R}(N(t+\varepsilon)-N(t))u(t)\phi(t)\d t.
\end{align*}
Since $\left(t\mapsto\left(\mathcal{N}_{\varepsilon,\rho}\partial_{t,\rho}u\right)(t)+\left(N(t+\varepsilon)-N(t)\right)u(t)\right)\in L_{2,\rho}(\R;H)$,
the claim follows from \cite[Proposition 4.1.1]{seifert2020evolutionary}. 

\item Let $\psi\in L_{2,\rho}(\R)$ and $x\in H$. Then 
\[
N_{\varepsilon}(t)x\psi(t)-N(t)x\psi(t)=\frac{1}{\varepsilon}\int_{t}^{t+\varepsilon}(N(s)-N(t))x\d s\:\psi(t)\to0\quad(\varepsilon\to0)
\]
for almost every $t\in\R$ by Lebesgue's differentiation theorem.
Moreover, 
\[
\|N_{\varepsilon}(t)x\psi(t)-N(t)x\psi(t)\|_{H}\leq2\|N\|_{\infty}\|\psi(t)x\|_{H}
\]
and thus, $\mathcal{N}_{\varepsilon}(\psi x)\to\mathcal{N}(\psi x)$
in $L_{2,\rho}(\R;H)$ by dominated convergence. Since 
\[
\|\mathcal{N}_{\varepsilon,\rho}\|_{L(L_{2,\rho}(\R;H))}=\|N_{\varepsilon}\|_{\infty}\leq\|N\|_{\infty}
\]
 for each $\varepsilon>0$, the strong convergence follows, since
$\lin\{\psi x\,;\,\psi\in L_{2,\rho}(\R),x\in H\}$ lies dense in
$L_{2,\rho}(\R;H)$. 

\item For $\phi\in H_{\rho}^{1/2}(\R;H)$ we estimate 
\begin{align*}
\|\partial_{t,\rho}^{1/2}\mathcal{N}_{\varepsilon,\rho}\phi\|_{\rho,0} & \leq\|\mathcal{N}_{\varepsilon,\rho}\partial_{t,\rho}^{1/2}\phi\|_{\rho,0}+c_{1}\|\phi\|_{\rho,1/2}+c_{2}\|\phi\|_{\rho,0}\\
 & \leq(\|N\|_{\infty}+c_{1})\|\phi\|_{\rho,1/2}+c_{2}\|\phi\|_{\rho,0}.
\end{align*}
Hence, the family $(\mathcal{N}_{\varepsilon,\rho}\phi)_{\varepsilon>0}$
is bounded in $H_{\rho}^{1/2}(\R;H)$ and thus, w.l.o.g. it converges
weakly in $H_{\rho}^{1/2}(\R;H)$ as $\varepsilon\to0$. Since $\mathcal{N}_{\varepsilon,\rho}\phi\to\mathcal{N}_{\rho}\phi$
in $L_{2,\rho}(\R;H)$ as $\varepsilon\to0$ by (b), we derive that
$\mathcal{N_{\mathbf{\rho}}}\phi\in H_{\rho}^{1/2}(\R;H)$ and $\partial_{t,\rho}^{1/2}\mathcal{N}_{\varepsilon,\rho}\phi\rightharpoonup\partial_{t,\rho}^{1/2}\mathcal{N_{\rho}}\phi$
in $L_{2,\rho}(\R;H).$ Hence, 
\[
\|[\partial_{t,\rho}^{1/2},\mathcal{N}_{\rho}]\phi\|_{\rho,0}\leq\lim_{\varepsilon\to0}\|[\partial_{t,\rho}^{1/2},\mathcal{N}_{\varepsilon,\rho}]\phi\|_{\rho,0}\leq c_{1}\|\phi\|_{\rho,1/2}+c_{2}\|\phi\|_{\rho,0}.\tag*{\qedhere}
\]

\end{enumerate}
\end{proof}
\begin{proof}[Proof of \prettyref{prop:Zacher}]
For the proof, we follow the rationale presented in \cite[Lemma 5.3]{Dier_Zacher2017}.
We first prove this assertion for the case $\mathcal{N}[H_{\rho}^{1}(\R;H)]\subseteq H_{\rho}^{1}(\R;H).$
Let $\phi\in H_{\rho}^{1}(\R;H).$ Then we get by \prettyref{lem:fractional_derivative}
\[
[\partial_{t,\rho}^{1/2},\mathcal{N}_{\rho}]\phi(t)=\frac{1}{2\Gamma(1/2)}\int_{-\infty}^{t}(t-s)^{-3/2}(N(t)-N(s))\phi(s)\d s\quad(t\in\R).
\]
Let now $0<\varepsilon\leq1$. Then
\begin{align*}
\|[\partial_{t,\rho}^{1/2},\mathcal{N_{\rho}}]\phi\|_{\rho,0} & =\frac{1}{2\Gamma(1/2)}\left(\int_{\R}\left\Vert \int_{-\infty}^{t}(t-s)^{-3/2}(N(t)-N(s))\phi(s)\d s\right\Vert _{H}^{2}\e^{-2\rho t}\d t\right)^{\frac{1}{2}}\\
 & \leq\frac{1}{2\Gamma(1/2)}\left(\int_{\R}\left\Vert \int_{t-\varepsilon}^{t}(t-s)^{-3/2}(N(t)-N(s))\phi(s)\d s\right\Vert _{H}^{2}\e^{-2\rho t}\d t\right)^{\frac{1}{2}}+\\
 & \quad+\frac{1}{2\Gamma(1/2)}\left(\int_{\R}\left\Vert \int_{-\infty}^{t-\varepsilon}(t-s)^{-3/2}(N(t)-N(s))\phi(s)\d s\right\Vert _{H}^{2}\e^{-2\rho t}\d t\right)^{\frac{1}{2}}.
\end{align*}
With the help of Young's inequality, the second integral can be estimated
by 
\begin{align*}
 & \left(\int_{\R}\left\Vert \int_{-\infty}^{t-\varepsilon}(t-s)^{-3/2}(N(t)-N(s))\phi(s)\d s\right\Vert _{H}^{2}\e^{-2\rho t}\d t\right)^{\frac{1}{2}}\\
 & \leq2\|N\|_{\infty}\|\left(t\mapsto\chi_{\R_{\geq\varepsilon}}(t)t^{-3/2}\right)\ast\|\phi(\cdot)\|_{H}\|_{\rho,0}\\
 & \leq2\|N\|_{\infty}c\|\phi\|_{\rho,0},
\end{align*}
where 
\[
c\coloneqq\int_{\varepsilon}^{\infty}s^{-3/2}\e^{-\rho s}\d s\leq\varepsilon^{-3/2}\frac{1}{\rho}.
\]
For estimating the first integral, let $\alpha\in]1-\delta,1[.$ Then
\begin{align*}
 & \left\Vert \int_{t-\varepsilon}^{t}(t-s)^{-3/2}(N(t)-N(s))\phi(s)\d s\right\Vert _{H}^{2}\\
 & \leq\int_{t-\varepsilon}^{t}(t-s)^{-\alpha}\d s\int_{t-\varepsilon}^{t}(t-s)^{-3+\alpha}\|N(t)-N(s)\|_{L(H)}^{2}\|\phi(s)\|_{H}^{2}\d s\\
 & =\frac{1}{1-\alpha}\varepsilon^{1-\alpha}\int_{t-\varepsilon}^{t}(t-s)^{-3+\alpha}\|N(t)-N(s)\|_{L(H)}^{2}\|\phi(s)\|_{H}^{2}\d s
\end{align*}
and thus, 
\begin{align*}
 & \left(\int_{\R}\left\Vert \int_{t-\varepsilon}^{t}(t-s)^{-3/2}(N(t)-N(s))\phi(s)\d s\right\Vert _{H}^{2}\e^{-2\rho t}\d t\right)^{\frac{1}{2}}\\
 & \leq\left(\frac{1}{1-\alpha}\varepsilon^{1-\alpha}\right)^{\frac{1}{2}}\left(\int_{\R}\int_{t-\varepsilon}^{t}(t-s)^{-3+\alpha}\|N(t)-N(s)\|_{L(H)}^{2}\|\phi(s)\|_{H}^{2}\d s\;\e^{-2\rho t}\d t\right)^{\frac{1}{2}}.
\end{align*}
Now we choose $p'\in]2,2\frac{2+\delta}{3-\alpha}]$ and $p\in]2,\infty[$
such that $\frac{1}{p}+\frac{1}{p'}=\frac{1}{2}.$ Then we apply the
Hölder inequality (see \prettyref{rem:FourierIP}(b)) and obtain 
\begin{align*}
 & \left(\int_{\R}\int_{t-\varepsilon}^{t}(t-s)^{-3+\alpha}\|N(t)-N(s)\|_{L(H)}^{2}\|\phi(s)\|_{H}^{2}\d s\;\e^{-2\rho t}\d t\right)^{\frac{1}{2}}\\
\leq & \left(\int_{\R}\int_{t-\varepsilon}^{t}\left((t-s)^{\left(-3+\alpha\right)/2}\|N(t)-N(s)\|_{L(H)}\right)^{2}\|\phi(s)\e^{-\rho s}\|_{H}^{2}\d s\;\d t\right)^{\frac{1}{2}}\\
\leq & \left(\int_{\R}\int_{t-\varepsilon}^{t}(t-s)^{\frac{\left(-3+\alpha\right)}{2}p'}\|N(t)-N(s)\|_{L(H)}^{p'}\d s\d t\right)^{\frac{1}{p'}}\left(\int_{\R}\int_{t-\varepsilon}^{t}\|\phi(s)\e^{-\rho s}\|_{H}^{p}\d s\d t\right)^{\frac{1}{p}}\\
= & \varepsilon^{\frac{1}{p}}\left(\int_{\R}\int_{t-\varepsilon}^{t}(t-s)^{\frac{\left(-3+\alpha\right)}{2}p'}\|N(t)-N(s)\|_{L(H)}^{p'}\d s\d t\right)^{\frac{1}{p'}}\|\phi\|_{\rho,p}.
\end{align*}
Since, on the one hand, $t-s\leq\varepsilon\leq1$ for $s\in[t-\varepsilon,t]$
and $\frac{\left(-3+\alpha\right)}{2}p'\geq-(2+\delta)$ and, on the
other hand, $\|N(t)-N(s)\|^{p'}\leq\|N(t)-N(s)\|^{2}\left(2\|N\|_{\infty}\right)^{p'-2}$,
we obtain the estimate 
\begin{align*}
 & \left(\int_{\R}\int_{t-\varepsilon}^{t}(t-s)^{\frac{\left(-3+\alpha\right)}{2}p'}\|N(t)-N(s)\|_{L(H)}^{p'}\d s\d t\right)^{\frac{1}{p'}}\\
 & \leq\left(\int_{\R}\int_{t-\varepsilon}^{t}(t-s)^{\frac{\left(-3+\alpha\right)}{2}p'+2+\delta}\frac{\|N(t)-N(s)\|_{L(H)}^{2}}{\left|t-s\right|^{2+\delta}}\left(2\|N\|_{\infty}\right)^{p'-2}\d s\d t\right)^{\frac{1}{p'}}\\
 & \leq(2\|N\|_{\infty})^{1-\frac{2}{p'}}C^{\frac{1}{p'}}.
\end{align*}
Summarising we have shown 
\begin{multline*}
\|[\partial_{t,\rho}^{1/2},\mathcal{N}]\phi\|_{\rho,0}\leq\frac{1}{2\Gamma(1/2)}\left(\frac{1}{1-\alpha}\varepsilon^{1-\alpha}\right)^{\frac{1}{2}}\varepsilon^{\frac{1}{p}}(2\|N\|_{\infty})^{1-\frac{2}{p'}}C^{\frac{1}{p'}}\|\phi\|_{\rho,p}\\+\frac{1}{2\Gamma(1/2)}2\|N\|_{\infty}\varepsilon^{-3/2}\frac{1}{\rho}\|\phi\|_{\rho,0}.
\end{multline*}
Under the assumption that $\mathcal{N}$ leaves $H_{\rho}^{1}(\R;H)$
invariant, the assertion follows from \prettyref{lem:Sobolev} and
the fact that $H_{\rho}^{1}(\R;H)$ is dense in $H_{\rho}^{1/2}(\R;H).$
\\
If $\mathcal{N}$ does not leave $H_{\rho}^{1}(\R;H)$ invariant,
we may replace it by $\mathcal{N}_{\varepsilon}$ as it is defined
in \prettyref{lem:regularise}. For each $\varepsilon>0$ we thus
obtain
\[
\|[\partial_{t,\rho}^{1/2},\mathcal{N}_{\varepsilon,\rho}]\phi\|_{\rho,0}\leq\left(\frac{1}{1-\alpha}\varepsilon^{1-\alpha}\right)^{\frac{1}{2}}\varepsilon^{\frac{1}{p}}(2\|N_{\varepsilon}\|_{\infty})^{1-\frac{2}{p'}}C_{\varepsilon}^{\frac{1}{p'}}\|\phi\|_{\rho,p}+2\|N_{\varepsilon}\|_{\infty}\varepsilon^{-3/2}\frac{1}{\rho}\|\phi\|_{\rho,0},
\]
where 
\begin{align*}
C_{\varepsilon} & \coloneqq\int_{\R}\int_{\R}\frac{\|N_{\varepsilon}(t)-N_{\varepsilon}(s)\|_{L(H)}^{2}}{|t-s|^{2+\delta}}\d t\d s\\
 & =\int_{\R}\int_{\R}\frac{\frac{1}{\varepsilon^{2}}\|\int_{0}^{\varepsilon}N(r+t)\d r-\int_{0}^{\varepsilon}N(r+s)\d r\|_{L(H)}^{2}}{(t-s)^{2+\delta}}\d s\d t\\
 & \leq\int_{\R}\int_{\R}\frac{1}{\varepsilon^{2}}\frac{\varepsilon\int_{0}^{\varepsilon}\|N(r+t)-N(r+s)\|_{L(H)}^{2}\d r}{(t-s)^{2+\delta}}\d s\d t\\
 & =\frac{1}{\varepsilon}\int_{0}^{\varepsilon}\int_{\R}\int_{\R}\frac{\|N(r+t)-N(r+s)\|_{L(H)}^{2}}{(t-s)^{2+\delta}}\d s\d t\d r=C.
\end{align*}
Since furthermore $\|N_{\varepsilon}\|_{\infty}\leq\|N\|_{\infty}$
by \prettyref{lem:regularise} (a), we can apply \prettyref{lem:regularise}
(c) and thus, the claim follows.
\end{proof}

\section{Examples\label{sec:Examples}}

\subsection{Divergence form equations\label{subsec:Divergence-form-equations}}

In order to treat a first standard example, we consider heat type
equations in this section and analyse the relationship to available
results in the literature. For this, we need to introduce the following
operators.
\begin{defn}
Let $\Omega\subseteq\mathbb{R}^{n}$ be open. We define
\begin{align*}
\grad_{(0)}\colon H_{(0)}^{1}(\Omega)\subseteq L_{2}(\Omega) & \to L_{2}(\Omega)^{n},\\
\phi & \mapsto\nabla\phi,
\end{align*}
and
\begin{align*}
\dive_{(0)}\colon H_{(0)}(\dive,\Omega)\subseteq L_{2}(\Omega)^{n} & \to L_{2}(\Omega),\\
\psi & \mapsto\nabla\cdot\psi,
\end{align*}
where $H^{1}(\Omega)$ is the standard Sobolev space of weakly differentiable
$L_{2}(\Omega)$ functions, $H_{0}^{1}(\Omega)$ the closure of $C_{c}^{\infty}(\Omega)$
in $H^{1}(\Omega)$. Similarly, $H(\dive,\Omega)$ is the space of
$L_{2}(\Omega)$-vector fields with distributional divergence in $L_{2}(\Omega)$
and $H_{0}(\dive,\Omega)$ is the closure of $C_{c}^{\infty}(\Omega)^{n}$
in $H(\dive,\Omega)$. 
\end{defn}

It is not difficult to see that $\dive_{0}^{*}=-\grad$ and $\grad_{0}^{*}=-\dive$,
see \cite[Chapter 6]{seifert2020evolutionary}. 

Next, we rephrase a sufficient criterion from \cite{Auscher_Egert2016},
which guarantees that $\left[\mathcal{N}_{0},\partial_{t,0}^{1/2}\right]$
is bounded. The result itself is a combination of the techniques used
in \cite{Auscher_Egert2016}, the BMO-characterisation by Strichartz
and the commutator estimate by Murray \cite{Murray1985}.
\begin{thm}
\label{thm:egertInt}Let $H=L_{2}(\Omega)^{n}$ for some open $\Omega\subseteq\mathbb{R}^{n}$,
$N\colon\mathbb{R}\times\Omega\to\mathbb{C}^{n\times n}$ measurable
and bounded. Assume there exists $C\geq0$ such that we have for a.e.~$x\in\Omega$
and for all intervals $I\subseteq\mathbb{R}$, 
\begin{equation}
\frac{1}{\ell(I)}\int_{I}\int_{I}\frac{\|N(t,x)-N(s,x)\|_{\mathbb{C}^{n\times n}}^{2}}{|t-s|^{2}}\d s\d t\leq C.\label{eq:egert}
\end{equation}
Then $[\mathcal{N}_{0},\partial_{t,0}^{1/2}]$ is bounded as an operator
in $L_{2}(\mathbb{R};H)$.
\end{thm}

\begin{proof}
A direct computation shows that $N_{\varepsilon}$ defined in \prettyref{lem:regularise}
satisfies the same condition imposed on $N$ in the present theorem
(with the same $C$). Thus, using \prettyref{lem:regularise} it suffices
to treat the case of Lipschitz continuous $N$. In this case, the
arguments in \cite[Corollary 7]{Auscher_Egert2016} show that both
$[\mathcal{N}_{0},\left|\partial_{t,0}\right|^{1/2}]$ and (using
\cite{Murray1985}) $[\mathcal{N}_{0},\sgn(-i\partial_{t,0})\left|\partial_{t,0}\right|^{1/2}]$
are bounded. Since 
\begin{align*}
\partial_{t,0}^{1/2} & =\chi_{[0,\infty)}(-i\partial_{t,0})\partial_{t,0}^{1/2}+\chi_{(-\infty,0)}(-i\partial_{t,0})\partial_{t,0}^{1/2}\\
 & =\frac{1}{2}\e^{\i\pi/4}\left(\left|\partial_{t,0}\right|^{1/2}+\sgn(-i\partial_{t,0})\left|\partial_{t,0}\right|^{1/2}\right)+\frac{1}{2}\e^{-\i\pi/4}\left(\left|\partial_{t,0}\right|^{1/2}-\sgn(-i\partial_{t,0})\left|\partial_{t,0}\right|^{1/2}\right)
\end{align*}
and due to linearity of the commutator in the second argument, we
infer the assertion.
\end{proof}
\begin{rem}
\label{rem:NinveEgert} Assume $N$ is given as in \prettyref{thm:egertInt}.
Moreover, assume that $\tilde{N}\colon(t,x)\mapsto N(t,x)^{-1}$ is
well-defined and bounded. Then $\tilde{N}$ satisfies the same integral
condition \prettyref{eq:egert} as $N$ does. Indeed, we compute for
a non-empty interval $I\subseteq\mathbb{R}$
\begin{align*}
 & \frac{1}{\ell(I)}\int_{I}\int_{I}\frac{\|\tilde{N}(t,x)-\tilde{N}(s,x)\|_{\mathbb{C}^{n\times n}}^{2}}{|t-s|^{2}}\d s\d t\\
 & =\frac{1}{\ell(I)}\int_{I}\int_{I}\frac{\|N(t,x)^{-1}-N(s,x)^{-1}\|_{\mathbb{C}^{n\times n}}^{2}}{|t-s|^{2}}\d s\d t\\
 & =\frac{1}{\ell(I)}\int_{I}\int_{I}\frac{\|N(t,x)^{-1}\left(N(s,x)-N(t,x)\right)N(s,x)^{-1}\|_{\mathbb{C}^{n\times n}}^{2}}{|t-s|^{2}}\d s\d t\\
 & \leq\|\tilde{N}\|_{\infty}^{4}\frac{1}{\ell(I)}\int_{I}\int_{I}\frac{\|\left(N(s,x)-N(t,x)\right)\|_{\mathbb{C}^{n\times n}}^{2}}{|t-s|^{2}}\d s\d t\leq\|\tilde{N}\|_{\infty}^{4}C.
\end{align*}
\end{rem}

At first we provide a proof of the main theorem in \cite{Auscher_Egert2016}
with the present tools.
\begin{thm}[{\cite[Theorem 2]{Auscher_Egert2016}}]
\label{thm:maxregEgert} Let $\Omega\subseteq\mathbb{R}^{n}$ be
open, $N\colon\mathbb{R}\times\Omega\to\mathbb{C}^{n\times n}$ measurable
and bounded, satisfying \prettyref{eq:egert}. Furthermore assume
that there exists $c>0$ such that for a.e. $(t,x)\in\mathbb{R}\times\Omega$:
\[
\Re\langle\xi,N(t,x)\xi\rangle_{\mathbb{C}^{n}}\geq c\|\xi\|_{\mathbb{C}^{n}}^{2}.
\]

Let $C\colon\dom(C)\subseteq L_{2}(\Omega)\to L_{2}(\Omega)^{n}$
be densely defined and closed such that $\grad_{0}\subseteq C\subseteq\grad,$
$\rho>0$ and let $f\in L_{2,\rho}(\mathbb{R};L_{2}(\Omega)).$ Then
the (unique) solution $u\in L_{2,\rho}(\mathbb{R};L_{2}(\Omega))$
of
\[
\partial_{t,\rho}u+C^{*}\mathcal{N}_{\rho}Cu=f
\]
admits maximal regularity, that is, 
\[
u\in H_{\rho}^{1}(\mathbb{R};L_{2}(\Omega))\cap H_{\rho}^{1/2}(\mathbb{R};\dom(C))\cap L_{2,\rho}(\mathbb{R};\dom(C^{*}\mathcal{N}_{\rho}C)).
\]
Moreover, the solution mapping $f\mapsto u$ is continuous as an operator
from $L_{2,\rho}(\mathbb{R};L_{2}(\Omega))$ into $H_{\rho}^{1}(\mathbb{R};L_{2}(\Omega))\cap H_{\rho}^{1/2}(\mathbb{R};\dom(C))\cap L_{2,\rho}(\mathbb{R};\dom(C^{*}\mathcal{N}_{\rho}C)).$
\end{thm}

\begin{proof}
In order to put ourselves into the framework of evolutionary equations,
we introduce the variable $q\coloneqq-\mathcal{N}_{\rho}Cu$ and consider
\begin{equation}
\left(\partial_{t,\rho}\left(\begin{array}{cc}
1 & 0\\
0 & 0
\end{array}\right)+\left(\begin{array}{cc}
0 & 0\\
0 & \mathcal{N}_{\rho}^{-1}
\end{array}\right)+\left(\begin{array}{cc}
0 & -C^{*}\\
C & 0
\end{array}\right)\right)\left(\begin{array}{c}
u\\
q
\end{array}\right)=\left(\begin{array}{c}
f\\
0
\end{array}\right).\label{eq:egertsystem}
\end{equation}
The latter equation is equivalent to 
\[
\partial_{t,\rho}u+C^{*}\mathcal{N}_{\rho}Cu=f\text{ and }q=-\mathcal{N}_{\rho}Cu.
\]
For a more detailed rationale on this we refer to \cite[Chapter 6]{seifert2020evolutionary}.
It is not diffcult to see that \prettyref{eq:egertsystem} satisfies
the well-posedness condition yielding unique existence of $(u,q)\in L_{2,\rho}(\mathbb{R};L_{2}(\Omega)^{1+n})$
, see e.g.~\cite{Picard2009,seifert2020evolutionary}. Moreover,
by \prettyref{thm:egertInt} together with \prettyref{rem:NinveEgert},
we infer that $\left[\mathcal{N}_{0}^{-1},\partial_{t,0}^{1/2}\right]$
is bounded in $L(L_{2}(\mathbb{R};L_{2}(\Omega)^{n})).$ Thus, \prettyref{prop:commutator_0}
leads to $\left[\mathcal{N}_{\rho}^{-1},\partial_{t,0}^{1/2}\right]$
being bounded in $L(L_{2,\rho}(\mathbb{R};L_{2}(\Omega)^{n})).$ Hence,
\prettyref{thm:max_reg_commuator} applies with $\tilde{c}=0$ to
\prettyref{eq:egertsystem}, which implies the assertion.
\end{proof}
\begin{rem}
(a) Using the extension result in \cite[Lemma 11]{Auscher_Egert2016}
and \prettyref{cor:local-in-time}, we obtain the corresponding local-in-time
result stated in \cite[Theorem 2]{Auscher_Egert2016}.

(b) The assumption of $C$ to be sandwiched inbetween $\grad_{0}$
and $\grad$ is the same assumption as in \cite{Auscher_Egert2016}
asking for either Dirichlet, Neumann or mixed boundary conditions.
\end{rem}

Next we provide our perspective on a main implication of the work
in \cite{Dier_Zacher2017} for homogeneous initial values.
\begin{thm}[{\cite[Corollary 1.1]{Dier_Zacher2017}}]
\label{thm:Zacher}Let $\Omega\subseteq\mathbb{R}^{n}$ be open,
$N\colon\mathbb{R}\to L(L_{2}(\Omega)^{n})$ strongly measurable and
bounded, satisfying \prettyref{eq:Zacher}. Furthermore assume that
there exists $c>0$ such that for a.e. $t\in\mathbb{R}$: 
\[
\Re\langle\xi,N(t)\xi\rangle_{L_{2}(\Omega)^{n}}\geq c\|\xi\|_{L_{2}(\Omega)^{n}}^{2}.
\]

Let $C\colon\dom(C)\subseteq L_{2}(\Omega)\to L_{2}(\Omega)^{n}$
be densely defined and closed such that $\grad_{0}\subseteq C\subseteq\grad,$
$\rho>0$ and let $f\in L_{2,\rho}(\mathbb{R};L_{2}(\Omega)).$ Then
the (unique) solution $u\in L_{2,\rho}(\mathbb{R};L_{2}(\Omega))$
of
\[
\partial_{t,\rho}u+C^{*}\mathcal{N}_{\rho}Cu=f
\]
admits maximal regularity, that is, 
\[
u\in H_{\rho}^{1}(\mathbb{R};L_{2}(\Omega))\cap H_{\rho}^{1/2}(\mathbb{R};\dom(C))\cap L_{2,\rho}(\mathbb{R};\dom(C^{*}\mathcal{N}_{\rho}C)).
\]
Moreover, the solution mapping $f\mapsto u$ is continuous as an operator
from $L_{2,\rho}(\mathbb{R};L_{2}(\Omega))$ into $H_{\rho}^{1}(\mathbb{R};L_{2}(\Omega))\cap H_{\rho}^{1/2}(\mathbb{R};\dom(C))\cap L_{2,\rho}(\mathbb{R};\dom(C^{*}\mathcal{N}_{\rho}C)).$
\end{thm}

\begin{proof}
Reformulating the equation in the variables $u$ and $q=-\mathcal{N}_{\rho}Cu,$
we obtain
\[
\left(\partial_{t,\rho}\left(\begin{array}{cc}
1 & 0\\
0 & 0
\end{array}\right)+\left(\begin{array}{cc}
0 & 0\\
0 & \mathcal{N}_{\rho}^{-1}
\end{array}\right)+\left(\begin{array}{cc}
0 & -C^{*}\\
C & 0
\end{array}\right)\right)\left(\begin{array}{c}
u\\
q
\end{array}\right)=\left(\begin{array}{c}
f\\
0
\end{array}\right).
\]
With the same argument as in \prettyref{rem:NinveEgert}, we infer
that $\tilde{N}\colon t\mapsto N(t)^{-1}$ satisfies \prettyref{eq:Zacher}.
Thus, with the help of \prettyref{prop:Zacher}, \prettyref{thm:max_reg_commuator}
is applicable, which yields the assertion. 
\end{proof}
\begin{rem}
(a) Even though the conditions \prettyref{eq:Zacher} and \prettyref{eq:egert}
do not compare (see \cite[Introduction]{Auscher_Egert2016}), we have
established that both of the results in \cite{Auscher_Egert2016}
and \cite{Dier_Zacher2017} applied to standard divergence form equations
can be obtained by the same overriding principle of suitably bounded
commutators with $\partial_{t,\rho}^{1/2}$. Note that \prettyref{eq:egert}
implies boundedness as an operator in $L_{2}$, whereas \prettyref{eq:Zacher}
yields infinitesimal boundedness relative to $\partial_{t,\rho}^{1/2}$
only.

(b) The condition on the regularity of the coefficient $N$ leading
to maximal regularity of the considered divergence form equation obtained
in \cite{Achache2019} seems to be weaker than the one of (infinitesimal)
boundedness of the commutator with $\partial_{t,\rho}^{1/2}$. However,
note that in order to apply the maximal regularity theorem in \cite{Achache2019},
one needs to assume Kato's square root property (potentially) resulting
in undue regularity requirements of the boundary of $\Omega$, which
we do not want to impose here.
\end{rem}

The above results (and the corresponding proofs) provide potential
for the following maximal regularity result, which invokes both lower
order terms and (time-)nonlocal effects in the time derivative term.
This will be addressed next.

\subsection{Maximal regularity for integro-differential equations\label{subsec:Maximal-regularity-integro}}

In this section, we consider equations of the following form (see
\cite{Trostorff2015} or \cite{Pruess1993}), which has applications
for instance in visco-elasticity. We need the following notion.
\begin{defn}
Let $G$ be a Hilbert space, $T\in L_{1,\mu}(\mathbb{R}_{\geq0};L(G))$
for some $\mu\geq0.$ We call $T$  \emph{admissible}, if the following
conditions are met:\begin{enumerate}

\item for all $t\geq0$, $T(t)$ is selfadjoint,

\item there exists $d\ge0$ and $\rho_{0}\geq\text{\ensuremath{\mu}}$
such that for all $t\in\mathbb{R}$ we have
\[
t\Im\hat{T}(t-\i\rho_{0})\leq d,
\]
where 
\[
\langle\hat{T}(t-\i\rho)\phi,\psi\rangle_{G}\coloneqq\frac{1}{\sqrt{2\pi}}\int\e^{-\i ts}\e^{-\rho s}\langle T(s)\phi,\psi\rangle_{G}\d s\quad(\phi,\psi\in G).
\]

\end{enumerate}
\end{defn}

\begin{rem}
As highlighted in \cite[Remark 3.6]{Trostorff2015}; $T$ being admissible
generalises the standard assumption for convolution kernels for the
class of integro-differential equations considered in the literature.
\end{rem}

\begin{prop}[{\cite[Proposition 3.9 (and its proof)]{Trostorff2015}}]
\label{prop:integroPD} Let $G$ be a Hilbert space, $T\in L_{1,\mu}(\mathbb{R}_{\geq0};L(G))$
for some $\mu\geq0.$ Assume that $T$ is admissible. Then there exists
$c_{1},c_{2}>0$, $\rho_{1}\geq\rho_{0}$ such that for all $\rho\geq\rho_{1}$
we have
\[
\Re\langle\partial_{t,\rho}(1+T*)\phi,\phi\rangle_{\rho,0}\geq\left(c_{1}\rho-c_{2}\right)\langle\phi,\phi\rangle_{\rho,0}\quad(\phi\in H_{\rho}^{1}(\mathbb{R};G)),
\]
where $T*\in L(L_{2,\rho}(\mathbb{R};G))$ is defined as the operator
of convolving with $T$ (extended by zero to $\mathbb{R}$).
\end{prop}

The corresponding theorem for maximal regularity of parabolic-type
non-autonomous integro-differential equations, now reads as follows.
\begin{thm}
Let $H_{0},H_{1}$ be Hilbert spaces, $\mu\geq0$, $T\in L_{1,\mu}(\mathbb{R};L(H_{0}))$
admissible. Assume $\mathcal{N}_{ij}\colon S_{c}(\mathbb{R};H_{j})\to\bigcap_{\rho\geq\rho_{0}}L_{2,\rho}(\mathbb{R};H_{i})$
is evolutionary at $\rho_{0}$ for each pair $(i,j)\in\{(0,0),(0,1),(1,1)\}$
and some $\rho_{0}$ satisfying 
\[
\Re\langle\mathcal{N}_{11}\phi_{1},\phi_{1}\rangle_{\rho,0}\geq c\langle\phi_{1},\phi_{1}\rangle_{\rho,0}\quad(\phi_{1}\in S_{c}(\mathbb{R};H_{1}),\rho\geq\rho_{0})
\]
for some $c>0$. In addition, assume that for all $\rho\geq\rho_{0}$
and all $\varepsilon>0$ there exists $d>0$ such that 
\[
\|\left[\mathcal{N}_{11,\rho},\partial_{t,\rho}^{1/2}\right]\phi_{1}\|_{\rho,0}\leq\varepsilon\|\phi_{1}\|_{\rho,\frac{1}{2}}+d\|\phi_{1}\|_{\rho,0}\quad(\phi_{1}\in H_{\rho}^{1/2}(\mathbb{R};H_{1})).
\]
Furthermore, let $C\colon\dom(C)\subseteq H_{0}\to H_{1}$ be densely
defined and closed. Then we find $\rho_{1}\geq\rho_{0}$ such that
for all $\rho\geq\rho_{1}$ and $f\in L_{2,\rho}(\mathbb{R};H_{0})$
there exists a unique $u\in L_{2,\rho}(\mathbb{R};H_{0})$ with 
\[
\partial_{t,\rho}(1+T*)u+\mathcal{N}_{00,\rho}u+\mathcal{N}_{01,\rho}\mathcal{N}_{11,\rho}^{-1}Cu+C^{*}\mathcal{N}_{11,\rho}^{-1}Cu=f.
\]
Moreover, $u$ satisfies the regularity 
\[
u\in H_{\rho}^{1}(\mathbb{R};H_{0})\cap H_{\rho}^{1/2}(\mathbb{R};\dom(C))\cap L_{2,\rho}(\mathbb{R};\dom(C^{*}\mathcal{N}_{11,\rho}^{-1}C)).
\]
\end{thm}

\begin{proof}
Using the substitution $q\coloneqq-\mathcal{N}_{11,\rho}^{-1}Cu$,
we consider for $\rho\geq\rho_{1}$ for $\rho_{1}\geq\rho_{0}$ to
be fixed later
\begin{equation}
\left(\partial_{t,\rho}\left(\begin{array}{cc}
(1+T*) & 0\\
0 & 0
\end{array}\right)+\left(\begin{array}{cc}
\mathcal{N}_{00,\rho} & -\mathcal{N}_{01,\rho}\\
0 & \mathcal{N}_{11,\rho}
\end{array}\right)+\left(\begin{array}{cc}
0 & -C^{*}\\
C & 0
\end{array}\right)\right)\left(\begin{array}{c}
u\\
q
\end{array}\right)=\left(\begin{array}{c}
f\\
0
\end{array}\right).\label{eq:integrosyst}
\end{equation}
For $\mathcal{M}=\left(\begin{array}{cc}
(1+T*) & 0\\
0 & 0
\end{array}\right)$ and $\mathcal{N}=\left(\begin{array}{cc}
\mathcal{N}_{00,\rho} & -\mathcal{N}_{01,\rho}\\
0 & \mathcal{N}_{11,\rho}
\end{array}\right)$, we show that the positive definiteness condition in \prettyref{thm:sol_theory_basic}
is satisfied. For this, we use \prettyref{prop:integroPD} and estimate
for $\phi=(\phi_{0},\phi_{1})\in L_{2,\rho}(\mathbb{R};H_{0}\times H_{1})$
and $\varepsilon>0$
\begin{align*}
 & \Re\langle\left(\partial_{t,\rho}\mathcal{M}+\mathcal{N}\right)\phi,\phi\rangle_{\rho,0}\\
 & \geq\left(c_{1}\rho-c_{2}\right)\langle\phi_{0},\phi_{0}\rangle_{\rho,0}-\|\mathcal{N}_{00,\rho}\|_{L(L_{2,\rho}(\mathbb{R};H_{0}))}\|\phi_{0}\|_{\rho,0}^{2}\\
 & \quad-\|\mathcal{N}_{01,\rho}\|_{L(L_{2,\rho}(\mathbb{R};H_{0}))}\|\phi_{0}\|_{\rho,0}\|\phi_{1}\|_{\rho,0}+c\|\phi_{1}\|_{\rho,0}^{2}\\
 & \geq\left(c_{1}\rho-c_{2}-\|\mathcal{N}_{00,\rho}\|_{L(L_{2,\rho}(\mathbb{R};H_{0}))}-\frac{1}{2\varepsilon}\|\mathcal{N}_{01,\rho}\|_{L(L_{2,\rho}(\mathbb{R};H_{0}))}^{2}\right)\|\phi_{0}\|_{\rho,0}^{2}+(c-\frac{1}{2}\varepsilon)\|\phi_{1}\|_{\rho,0}^{2}.
\end{align*}
Thus, choosing $\varepsilon>0$ small enough, we find $\rho_{1}\geq\rho_{0}$
such that for all $\rho\geq\rho_{1}$ we have
\[
\Re\langle\left(\partial_{t,\rho}\mathcal{M}+\mathcal{N}\right)\phi,\phi\rangle_{\rho,0}\geq\frac{c}{2}\|\phi\|_{\rho,0}^{2}.
\]
Hence, $(u,q)\in L_{2,\rho}(\mathbb{R};H_{0}\times H_{1})$ are uniquely
determined by \prettyref{eq:integrosyst}. Note that \prettyref{thm:sol_theory_basic}
asserts that actually 
\[
\overline{\left(\partial_{t,\rho}\left(\begin{array}{cc}
(1+T*) & 0\\
0 & 0
\end{array}\right)+\left(\begin{array}{cc}
\mathcal{N}_{00,\rho} & -\mathcal{N}_{01,\rho}\\
0 & \mathcal{N}_{11,\rho}
\end{array}\right)+\left(\begin{array}{cc}
0 & -C^{*}\\
C & 0
\end{array}\right)\right)}\left(\begin{array}{c}
u\\
q
\end{array}\right)=\left(\begin{array}{c}
f\\
0
\end{array}\right).
\]
Since both $\mathcal{N}_{00,\rho}$ and $\mathcal{N}_{01,\rho}$ are
bounded linear operators, we obtain
\begin{align*}
 & \overline{\left(\partial_{t,\rho}\left(\begin{array}{cc}
(1+T*) & 0\\
0 & 0
\end{array}\right)+\left(\begin{array}{cc}
0 & 0\\
0 & \mathcal{N}_{11,\rho}
\end{array}\right)+\left(\begin{array}{cc}
0 & -C^{*}\\
C & 0
\end{array}\right)\right)}\left(\begin{array}{c}
u\\
q
\end{array}\right)\\
 & =\left(\begin{array}{c}
f\\
0
\end{array}\right)+\left(\begin{array}{cc}
-\mathcal{N}_{00,\rho} & \mathcal{N}_{01,\rho}\\
0 & 0
\end{array}\right)\left(\begin{array}{c}
u\\
q
\end{array}\right)\\
 & =\left(\begin{array}{c}
f-\mathcal{N}_{00,\rho}u+\mathcal{N}_{01,\rho}q\\
0
\end{array}\right).
\end{align*}
Next, as the convolution operator $(1+T*)$ commutes with $\partial_{t,0}^{-1}$,
we infer with the help of \prettyref{thm:max_reg_commuator} the desired
regularity statement.
\end{proof}
\begin{rem}
(a) Note that the coefficients of the lower order terms $\mathcal{N}_{01}$
and $\mathcal{N}_{00}$ are not required to satisfy any regularity
in time, which is in line with the concluding example in \cite{Achache2019}.
Moreover, in the theorem presented here the coefficient $\mathcal{N}_{11,\rho}$
may well depend suitably regular on time, i.e., $\mathcal{N}_{11,\rho}$
may be induced by a multiplication operator, which satisfies either
\prettyref{eq:egert} or \prettyref{eq:Zacher}. 

(b) The results above directly apply to systems of divergence form
equations, see \cite[Parabolic systems]{Dier_Zacher2017} for examples
concerning maximal regularity and \cite[Proposition 3.8]{CW17_FH}
for the corresponding formulation as evolutionary equation.
\end{rem}

\subsection{Maxwell's equations\label{subsec:Maxwell's-equations}}

The concluding example is concerned with Maxwell's equations. For
this, we introduce the necessary operator from vector analysis:
\begin{defn}
Let $\Omega\subseteq\mathbb{R}^{3}$ open. Then we define
\begin{align*}
\curl_{(0)}\colon H_{(0)}(\curl,\Omega)\subseteq L_{2}(\Omega)^{3} & \to L_{2}(\Omega)^{3},\\
\phi & \mapsto\nabla\times\phi,
\end{align*}
where $H(\curl,\Omega)$ is the space of $L_{2}(\Omega)$-vector fields
with distributional $\curl$ in $L_{2}(\Omega)^{3}$ and $H_{0}(\curl,\Omega)$
is the closure of $C_{c}^{\infty}(\Omega)^{3}$ in $H(\curl,\Omega)$.
It is not difficult to see that $\curl_{0}^{*}=\curl.$ 

The result on maximal regularity for Maxwell's equations is concerned
with the eddy current approximation, which is a parabolic variant
of the originial Maxwell's equations. The catch is that in electrically
conducting materials like metals the dielectricity $\varepsilon$
is negligible compared to the conductivity $\sigma$, which we assume
to depend on time. This setting has applications to moving domains,
see e.g.~\cite{Cooper1985}. The result reads as follows.
\end{defn}

\begin{thm}
\label{thm:Maxwellmax}Let $\Omega\subseteq\mathbb{R}^{3}$ open,
$\rho>0$ and let $\mu=\mu^{*}\in L(L_{2}(\Omega)^{3})$; assume $\mu\geq c$
for some $c>0$ . Moreover, let $\sigma\in L(L_{2,\rho}(\mathbb{R};L_{2}(\Omega)^{3})$
satisfy
\[
\Re\langle\sigma E,E\rangle_{\rho,0}\geq c\|E\|_{\rho,0}^{2}\quad(E\in L_{2,\rho}(\mathbb{R};L_{2}(\Omega)^{3}).
\]
and for all $\varepsilon>0$ we find $d>0$ such that 
\[
\|\left[\sigma,\partial_{t,\rho}^{1/2}\right]\phi\|\leq\varepsilon\|\phi\|_{\rho,\frac{1}{2}}+d\|\phi\|_{\rho,0}\quad(\phi\in H_{\rho}^{1/2}(\mathbb{R};L_{2}(\Omega)^{3})).
\]
Then for all $\rho>0$ the equation
\[
\left(\partial_{t,\rho}\left(\begin{array}{cc}
\mu & 0\\
0 & 0
\end{array}\right)+\left(\begin{array}{cc}
0 & 0\\
0 & \sigma
\end{array}\right)+\left(\begin{array}{cc}
0 & \curl_{0}\\
-\curl & 0
\end{array}\right)\right)\left(\begin{array}{c}
H\\
E
\end{array}\right)=\left(\begin{array}{c}
K\\
-J
\end{array}\right)
\]
admits a unique solution $(E,H)\in L_{2,\rho}(\mathbb{R};L_{2}(\Omega)^{6})$
for all $(J,K)\in L_{2,\rho}(\mathbb{R};L_{2}(\Omega)^{6}).$ If,
in addition, $J\in H_{\rho}^{1/2}(\mathbb{R};L_{2}(\Omega)^{3})$
then
\begin{align*}
E & \in H_{\rho}^{1/2}(\mathbb{R};L_{2}(\Omega)^{3})\cap L_{2,\rho}(\mathbb{R};H_{0}(\curl,\Omega))\text{ and }\\
H & \in H_{\rho}^{1}(\mathbb{R};L_{2}(\Omega)^{3})\cap H_{\rho}^{1/2}(\mathbb{R};H(\curl,\Omega)).
\end{align*}
\end{thm}

\begin{proof}
The proof is a direct application of \prettyref{thm:max_reg_commuator}.
\end{proof}
\begin{rem}
(a) Again, the commutator condition imposed on $\sigma$ is satisfied,
if $\sigma$ is a multiplication operator induced by a function satisfying
either \prettyref{eq:egert} or \prettyref{eq:Zacher}. 

(b) In applications, non-zero terms $K$ can occur, if one considers
inhomogeneous boundary values. Note that a result corresponding to
\prettyref{thm:Maxwellmax} is valid also for mixed boundary conditions
or with homogeneous boundary conditions for $H$. 

(c) There is no condition assumed on the regularity of the boundary
of $\Omega$.
\end{rem}

\section{Conclusion\label{sec:Conclusion}}

We presented a maximal regularity theorem for evolutionary equations.
The core assumptions abstracly describe a parabolic type evolutionary
equation and lead to well-posedness on $L_{2,\rho}$ and $H_{\rho}^{1/2}$.
For applications, the operator theoretic insight of the need of commutator
estimates for the commutator with $\partial_{t}^{1/2}$ found in \cite{Dier_Zacher2017}
and \cite{Auscher_Egert2016} showed to be decisive also for evolutionary
equations. Moreover, we showed that both conditions on the coefficients
imposed in \cite{Dier_Zacher2017} and \cite{Auscher_Egert2016},
which are not comparable, imply the well-posedness in $H_{\rho}^{1/2}$
and hence, yield the maximal regularity of the problem under consideration
within the presented framework. Naturally, the regularity phenomenon
for the unknown to belong to $H^{1/2}$ with values in the form domain,
observed in \cite{Dier_Zacher2017} and \cite{Auscher_Egert2016},
resurfaced also in the framework of evolutionary equations. The conditions
derived here are deliberately focussed on the coefficients rather
than the whole space-time operator in order that it is possible to
generate results independent of the regularity of the boundary of
the underlying domain, which is needed in \cite{Achache2019} in order
to warrant some form of the square root property. Due to the view
of the time derivative as a normal continuously invertible operator
it is possible to use a straightforward functional calculus and to
compute fractional powers of the time derivative and to work with
them without the need of explicitly invoking the Hilbert transform
or other technicalities. It remains to be seen, whether the commutator
assumptions or the basic result \prettyref{thm:main} implying maximal
regularity lead to slightly stronger statements also in the situation
of divergence form equations.

\bibliographystyle{abbrv}

\begin{thebibliography}{10}

\bibitem{Achache2019}
M.~Achache and E.~M. Ouhabaz.
\newblock Lions' maximal regularity problem with {$H^{\frac12}$}-regularity in
  time.
\newblock {\em J. Differential Equations}, 266(6):3654--3678, 2019.

\bibitem{Arendt2017}
W.~Arendt, D.~Dier, and S.~Fackler.
\newblock J. {L}. {L}ions' problem on maximal regularity.
\newblock {\em Arch. Math. (Basel)}, 109(1):59--72, 2017.

\bibitem{Auscher_Egert2016}
P.~Auscher and M.~Egert.
\newblock On non-autonomous maximal regularity for elliptic operators in
  divergence form.
\newblock {\em Arch. Math. (Basel)}, 107(3):271--284, 2016.

\bibitem{Bergh_Lofstrom1976}
J.~Bergh and J.~L\"{o}fstr\"{o}m.
\newblock {\em Interpolation spaces. {A}n introduction}.
\newblock Springer-Verlag, Berlin-New York, 1976.
\newblock Grundlehren der Mathematischen Wissenschaften, No. 223.

\bibitem{Beyer2007}
H.~Beyer.
\newblock {\em {Beyond partial differential equations. On linear and
  quasi-linear abstract hyperbolic evolution equations.}}
\newblock Springer, Berlin, 2007.

\bibitem{Cooper1985}
J.~Cooper and W.~Strauss.
\newblock The initial boundary problem for the {M}axwell equations in the
  presence of a moving body.
\newblock {\em SIAM J. Math. Anal.}, 16(6):1165--1179, 1985.

\bibitem{CW17_FH}
S.~Cooper and M.~Waurick.
\newblock Fibre homogenisation.
\newblock {\em Journal of Functional Analysis}, 276(11):3363--3405, 2019.

\bibitem{Dier_Zacher2017}
D.~Dier and R.~Zacher.
\newblock Non-autonomous maximal regularity in {H}ilbert spaces.
\newblock {\em J. Evol. Equ.}, 17(3):883--907, 2017.

\bibitem{Diethelm2019}
K.~Diethelm, K.~Kitzing, R.~Picard, S.~Siegmund, S.~Trostorff, and M.~Waurick.
\newblock {A Hilbert space approach to fractional differential equations}.
\newblock Technical report, 2019.

\bibitem{Fackler2017}
S.~Fackler.
\newblock J.-{L}. {L}ions' problem concerning maximal regularity of equations
  governed by non-autonomous forms.
\newblock {\em Ann. Inst. H. Poincar\'{e} Anal. Non Lin\'{e}aire},
  34(3):699--709, 2017.

\bibitem{Lions1961}
J.-L. Lions.
\newblock {\em \'{E}quations diff\'{e}rentielles op\'{e}rationnelles et
  probl\`emes aux limites}.
\newblock Die Grundlehren der mathematischen Wissenschaften, Bd. 111.
  Springer-Verlag, Berlin-G\"{o}ttingen-Heidelberg, 1961.

\bibitem{Lunardi2018}
A.~Lunardi.
\newblock {\em Interpolation theory}, volume~16 of {\em Appunti. Scuola Normale
  Superiore di Pisa (Nuova Serie) [Lecture Notes. Scuola Normale Superiore di
  Pisa (New Series)]}.
\newblock Edizioni della Normale, Pisa, 2018.
\newblock Third edition.

\bibitem{Murray1985}
M.~A.~M. Murray.
\newblock Commutators with fractional differentiation and {BMO} {S}obolev
  spaces.
\newblock {\em Indiana Univ. Math. J.}, 34(1):205--215, 1985.

\bibitem{Ouhabaz2010}
E.~M. Ouhabaz and C.~Spina.
\newblock Maximal regularity for non-autonomous {S}chr\"{o}dinger type
  equations.
\newblock {\em J. Differential Equations}, 248(7):1668--1683, 2010.

\bibitem{Picard2009}
R.~Picard.
\newblock A structural observation for linear material laws in classical
  mathematical physics.
\newblock {\em Math. Methods Appl. Sci.}, 32(14):1768--1803, 2009.

\bibitem{Picard_McGhee2011}
R.~Picard and D.~McGhee.
\newblock {\em Partial differential equations}, volume~55 of {\em De Gruyter
  Expositions in Mathematics}.
\newblock Walter de Gruyter GmbH \& Co. KG, Berlin, 2011.
\newblock A unified Hilbert space approach.

\bibitem{Picard2020}
R.~Picard, D.~McGhee, S.~Trostorff, and M.~Waurick.
\newblock {\em {A Primer for a Secret Shortcut to PDEs of Mathematical
  Physics}}.
\newblock {Frontiers in Mathematics}. Birkh\"auser, 2020.
\newblock {manuscript accepted; contract signed, approx 200 pp}.

\bibitem{PTW2015_fractional}
R.~Picard, S.~Trostorff, and M.~Waurick.
\newblock On evolutionary equations with material laws containing fractional
  integrals.
\newblock {\em Math. Methods Appl. Sci.}, 38(15):3141--3154, 2015.

\bibitem{PTW2015_survey}
R.~Picard, S.~Trostorff, and M.~Waurick.
\newblock Well-posedness via monotonicity---an overview.
\newblock In {\em Operator semigroups meet complex analysis, harmonic analysis
  and mathematical physics}, volume 250 of {\em Oper. Theory Adv. Appl.}, pages
  397--452. Birkh\"{a}user/Springer, Cham, 2015.

\bibitem{PTW2017_maxreg}
R.~Picard, S.~Trostorff, and M.~Waurick.
\newblock On maximal regularity for a class of evolutionary equations.
\newblock {\em J. Math. Anal. Appl.}, 449(2):1368--1381, 2017.

\bibitem{PTWW13_NA}
R.~Picard, S.~Trostorff, M.~Waurick, and M.~Wehowski.
\newblock {On Non-Autonomous Evolutionary Problems}.
\newblock {\em {Journal of Evolution Equations}}, 13:751--776, 2013.

\bibitem{Pruess1993}
J.~Pr\"{u}ss.
\newblock {\em Evolutionary integral equations and applications}.
\newblock Modern Birkh\"{a}user Classics. Birkh\"{a}user/Springer Basel AG,
  Basel, 1993.
\newblock [2012] reprint of the 1993 edition.

\bibitem{seifert2020evolutionary}
C.~Seifert, S.~Trostorff, and M.~Waurick.
\newblock Evolutionary equations, 2020.
\newblock Lecture notes of the 23rd Internet seminar, arXiv: 2003.12403.

\bibitem{Trostorff2015}
S.~Trostorff.
\newblock On integro-differential inclusions with operator-valued kernels.
\newblock {\em Math. Methods Appl. Sci.}, 38(5):834--850, 2015.

\bibitem{Trostorff2020a}
S.~Trostorff.
\newblock Well-posedness for a general class of differential inclusions.
\newblock {\em J. Differential Equations}, 268(11):6489--6516, 2020.

\bibitem{W15_CB}
M.~Waurick.
\newblock {A note on causality in Banach spaces}.
\newblock {\em {Indagationes Mathematicae}}, 26(2):404--412, 2015.

\bibitem{Waurick2015}
M.~Waurick.
\newblock On non-autonomous integro-differential-algebraic evolutionary
  problems.
\newblock {\em Math. Methods Appl. Sci.}, 38(4):665--676, 2015.

\bibitem{W16_H}
M.~Waurick.
\newblock {\em {On the continuous dependence on the coefficients of
  evolutionary equations}}.
\newblock Habilitation, TU Dresden, 2016.
\newblock {arXiv:1606.07731}.

\end{thebibliography}

\end{document}